\documentclass[12pt,leqno]{article}
\usepackage{graphicx, fullpage}
\usepackage{amsmath,amssymb,amsthm,amscd, bm}
\usepackage{fancyhdr}
\usepackage[mathscr]{eucal}
\usepackage{amsfonts}
\begin{document}
\parskip=6pt

\theoremstyle{plain}

\newtheorem {thm}{Theorem}[section]
\newtheorem {lem}[thm]{Lemma}
\newtheorem {cor}[thm]{Corollary}
\newtheorem {defn}[thm]{Definition}
\newtheorem {prop}[thm]{Proposition}
\numberwithin{equation}{section}

\newcommand{\cF}{{\cal F}}
\newcommand{\cA}{{\cal A}}
\newcommand{\cC}{{\cal C}}
\newcommand{\cH}{{\cal H}}
\newcommand{\cU}{{\cal U}}
\newcommand{\cK}{{\cal K}}
\newcommand{\cM}{{\cal M}}
\newcommand{\cO}{{\cal O}}
\newcommand{\cN}{{\cal N}}
\newcommand{\cV}{{\cal V}}
\newcommand{\cE}{{\cal E}}
\newcommand{\bC}{\mathbb C}
\newcommand{\bP}{\mathbb P}
\newcommand{\bN}{\mathbb N}
\newcommand{\bA}{\mathbb A}
\newcommand{\bR}{\mathbb R}
\newcommand{\bZ}{\mathbb Z}
\newcommand{\bT}{\mathbb T}
\newcommand{\fg}{\mathfrak g}
\newcommand{\fy}{\mathfrak y}
\newcommand{\fh}{\mathfrak h}
\newcommand{\fV}{\mathfrak V}
\newcommand{\var}{\varepsilon}
\renewcommand\qed{ }
\newcommand{\sgrad}{\text{sgrad} \,}
\newcommand{\sgru}{\text{sgrad}}
\newcommand{\grad}{\text{grad}\,}
\newcommand{\gru}{\text{grad}}
\newcommand{\id}{\text{id}}
\newcommand{\ad}{\text{ad}}
\newcommand{\const}{\text{const}\,}
\newcommand{\Ker}{\text{Ker}\,}
\newcommand{\opartial}{\overline\partial}
\newcommand{\Ree}{\text{Re}\,}
\newcommand{\an}{\text{an}}
\newcommand{\tr}{\text{tr}}
\newcommand{\wf}{\text{WF}_{\text A}}
\newcommand{\codim}{\text{codim}\,}

\begin{titlepage}
\title{\bf Isometries in spaces of K\"ahler potentials\thanks{Research partially supported by NSF grant DMS--1464150 and by the Center for Advanced Study of the Norwegian Academy of Sciences.\newline 2010 Mathematics subject classification 32Q15, 53C55}}
\author{L\'aszl\'o Lempert\\ Department of  Mathematics\\
Purdue University\\West Lafayette, IN
47907-2067, USA}
\end{titlepage}
\date{}
\maketitle
\abstract

The space of K\"ahler potentials in a compact K\"ahler manifold, endowed with Mabu\-chi's metric, is an infinite dimensional Riemannian manifold.
We characterize local isometries between spaces of K\"ahler potentials, and prove existence and uniqueness for such isometries.

\section{Introduction}

Let $X$ be an $n$ dimensional, connected compact complex manifold, and $\omega$ a smooth K\"ahler form on it.
Denoting by $C^\infty(X)$ the space of real valued smooth functions on $X$, the space of (relative) K\"ahler potentials is
$$
\cH=\cH(\omega)=\{u\in C^\infty(X)\colon \omega_u=\omega+i\partial\overline\partial u > 0\},
$$
a convex open subset of $C^\infty(X)$.
The space $\cH$ inherits a Fr\'echet manifold structure from $C^\infty(X)$, and each tangent space $T_u\cH$ can be canonically identified with
$C^\infty(X)$ itself.
The Mabuchi length $|\xi|_u$ of a tangent vector $\xi\in T_u\cH\approx C^\infty (X)$ is given by
$$
|\xi|_u^2=\int_X \xi^2 \omega_u^n \bigg/\int_X \omega^n.
$$
This turns $\cH$ into a smooth Riemannian manifold (see [L2, section 5 and Example 2.1]), whose study was initiated by Mabuchi, Semmes, and
Donaldson \cite{Do1,M,Se}.

The curvature of $\cH$ is covariantly constant, a property that for finite dimensional manifolds would imply the existence of local symmetries,
self--isometries of neighborhoods of an arbitrary point, that act on tangent vectors issued from the point by $-\text{id}$.
In \cite{Se} Semmes indicated how to obtain such local symmetries, even in (a variant of) $\cH$, about any $u$ for which $\omega_u$ is analytic.
More recently, in \cite{BCKR} Berndtsson, Cordero--Erausquin, Klartag, and Rubinstein constructed these symmetries by generalizing the Legendre
transformation from Euclidean spaces to K\"ahler manifolds.

Our goal here is to explore existence, uniqueness, and other properties of local isometries in spaces of K\"ahler potentials.
In particular, we will see that the symmetry property uniquely determines the local isometries constructed in 
\cite{Se, BCKR}, cf. Theorem 1.3 below.

Before proceeding, we need to discuss topologies on 
$\cH\subset C^\infty(X)$ and on $T_u\cH\approx C^\infty(X)$, of which there are several. For each $k=0,1,\dots,\infty$ there 
are $C^k$ topologies on $C^\infty(X)$, hence on $\cH$ and on $T_u\cH$; there is also the topology on $\cH$ induced by 
Mabuchi's metric and the corresponding topology on $T_u\cH$. In this paper topological notions like `openness', `density', 
etc. typically refer to the $C^\infty$ topology. The only exceptions are Theorem 1.4, and the corresponding 
Section 6, where we talk about $C^2$ neighborhoods; and Section 7, where topologies on real analytic 
functions also appear.

The main results are as follows.
Fix another connected compact K\"ahler manifold $(X',\omega')$.
We will write $\cH'=\cH(\omega')$ for its space of relative K\"ahler potentials; whenever $v\in C^\infty (X')$, we will abbreviate
$\omega'+i\partial\overline\partial v$ as $\omega_v$, and will forgo the more pedantic notation $\omega'_v$.
$|\quad|$ will denote Mabuchi length in $\cH'$ as well.
If $\cU\subset\cH$ is open and $k=1,2,\ldots$, by a $C^k$ isometry $F\colon\cU\to\cH'$ we mean a $C^k$ diffeomorphism between $\cU$ and an open $\cU'\subset\cH'$
whose differential $F_*\colon T\cU\to T\cH'$ satisfies $|F_*\xi|_{F(u)}=|\xi|_u$ for all $u\in\cU$ and $\xi\in T_u\cU$. \cite[Corollary 5.2]{L2} implies that  $C^2$ isometries are automatically $C^\infty$, but we do not know whether mere $C^1$ isometries need to be smooth.

\begin{thm}Suppose $\cU\subset\cH$ is open, $u\in\cU$, and $F\colon\cU\to\cH'$ is a $C^1$ isometry.
Then there are a $C^\infty$ diffeomorphism $\varphi\colon X'\to X$ and real numbers $a=\pm 1$, $b=0\text{ or }2a/\int_X\omega^n$
such that $\varphi^*\omega_u=\pm\omega_{F(u)}$ and
$$
F_* \xi=a\varphi^*\xi-b\int_X\xi \omega_u^n,\qquad\xi\in T_u\cH\approx C^\infty (X).
$$
In particular, the Poisson brackets $\{\ , \ \}_u, \{ \ , \ \}_{F(u)}$ determined by the symplectic forms $\omega_u,\omega_{F(u)}$ satisfy $\{F_*\xi,
F_*\eta\}_{F(u)}=\pm  F_*\{\xi,\eta\}_u$.
\end{thm}

The question arises whether $\varphi$ of the theorem depends on $u\in\cU$. Darvas in \cite{Da3} proves that it does not if 
$(X,\omega)=(X',\omega')$ and $F$ is a global isometry $\cH\to\cH$. However, for the local isometries $F$ we construct in 
Section 6 $\varphi$ typically varies with $u$. This can be read off from Theorem 6.6, that provides a formula for $F_*|T_u\cH$. 

Mabuchi pointed out that $\cH$ is a Riemannian product $\bR\times\cH_0$.
On the level of tangent spaces, the splitting is given by 
$$
T_u\cH=\{\xi\colon\xi=\text{ const}\}\oplus \{\xi\colon\int_X\xi\omega_u^n=0\}.
$$
By Theorem 1.1 isometries respect this splitting. Another consequence of Theorem 1.1 is that isometries of Mabuchi's 
metric with $b=0$ preserve all the Orlicz--type Finsler metrics on $\cH$ that Darvas constructed in \cite{Da2}.

\begin{thm}With $u\in\cU$ and $F$ as in Theorem 1.1, if $\omega_u$ is real analytic, then so is $\omega_{F(u)}$, and  $F_*|T_u\cU$ maps real analytic functions to real analytic functions. (Equivalently, the diffeomorphism
$\varphi\colon X'\to X$ of Theorem 1.1 is real analytic.)

\end{thm}

\begin{thm}Suppose $\cU\subset\cH$ is open and connected, $u\in\cU$, and $F,G\colon\cU\to\cH'$ are $C^\infty$ isometries.
If $\omega_u$ is real analytic, $F(u)=G(u)$, and $F_*,G_*$ agree on $T_u\cH$, then $F=G$.
\end{thm}

This theorem, unlike the previous two, applies only to $C^\infty$ isometries. It could be extended to $C^1$ isometries if Theorem 3.3 further down could be extended; or if Lemma 5.1 could be proved for $C^1$ curves (or even only for $C^\infty$ curves).

\begin{thm}Let $u\in\cH$, $u'\in\cH'$ be such that $\omega_u$, $\omega_{u'}$ are real analytic, and let $\Phi\colon T_u\cH\to T_{u'}\cH'$ be an
isomorphism of vector spaces.
Suppose that $|\Phi\xi|_{u'}=|\xi|_u$ for all $\xi\in T_u\cH$, 
$$
\{\Phi\xi,\Phi\eta\}_{u'}=\pm \Phi\{\xi,\eta\}_u\qquad\text{for all }\xi,\eta\in T_u\cH\approx C^\infty(X),
$$
and $\Phi$ maps real analytic functions in $C^\infty(X)$ to real analytic functions in $C^\infty(X')$.
Then there is a $C^\infty$ isometry $F\colon\cU\to\cH'$ of some neighborhood $\cU$ of $u$ such that $F(u)=u'$ and $F_*|T_u\cU=\Phi$.
This $\cU$ can be chosen a neighborhood in the $C^2$ topology.
\end{thm}

In light of Theorems 1.1, 1.2 the sufficient conditions on $\Phi$ are also necessary.

Semmes in \cite{Se} produced related transformations when $(X',\omega')=(X,\omega)$.
One difference is that he transforms closed 1--forms rather than functions.
His transformations can be used to transform K\"ahler forms, and to produce  isometries of K\"ahler potentials through a Riemannian splitting $\cH=\bR\times \{$K\"ahler forms $\sim\omega\}$.
Instead of going that route, though, we will construct the isometries $F$ directly from the transformations we introduced in \cite{L1}.

The prominent role that analytic K\"ahler forms play in the theorems above suggests to consider the subspaces of $\cH,\cH'$ consisting of relative
potentials of such forms.
In any K\"ahler class analytic K\"ahler forms are dense, see Proposition 2.1.
If we let 
$$
\cK=\cK(\omega)=\{u\in\cH\colon\omega_u\text{ is analytic}\}
$$
and similarly $\cK'\subset\cH'$, then the theorems above hold with $\cK,\cK'$ replacing $\cH,\cH'$.
The analytic version of Theorem 1.2 is tautological; we will discuss the analytic versions of the rest in section 7.

The initial idea of this work is that isometries map geodesics to geodesics. This is quite obvious, on formal grounds, for $C^2$ isometries, but for $C^1$ isometries it lies deeper, see Lemma 3.1. We then show that various notions in $\cH$ can be expressed in terms of geodesics: curvature, parallel transport, and somewhat surprisingly, regularity of maps into $\cH$. We deal with these issues in sections 4, 5, and 2, 3, respectively. Putting all this together results in enough structure for isometries to prove Theorems 1.1, 1.2, and 1.3, in sections 4, 3, and 5. Theorem 1.4 is proved in section 6, along different lines.

It would be interesting to clarify which among the local isometries extend to global isometries $\cH\to\cH'$. We conjecture 
that this happens only very exceptionally.\footnote{Darvas has recently answered this question when $\cH=\cH'$ 
in \cite{Da3}.} 
But an even more interesting question is, what are the isometries between the metric completions of $\cH,\cH'$ that Darvas constructed in \cite{Da1}? The results we formulated raise a few more natural questions. 
Must $C^1$ isometries be automatically $C^\infty$? Does uniqueness, Theorem 1.3, hold at non--analytic $\omega_u$? Suppose $u\in\cH$ is the fixed point of a symmetry $F\colon \cU\to\cH$, i.e., of an isometry such that $F_*|T_u\cH=-\text{id}$. Must $\omega_u$ be analytic?\footnote{In the meantime we proved that this is so, see \cite{L3}.}

In our proofs we will use infinite dimensional variants of basic facts of Riemannian geometry.
We invite the reader who is reluctant to accept these facts on faith to refer to \cite{L2}, where we collected background material in infinite
dimensional Riemannian geometry.

A note on notation: if $M$ is a manifold, $C^k(M)$ stands for the space of real valued $C^k$ functions, except in parts of section 2, where the
reader will be warned.
If $M$ is a real analytic manifold, we write $C^{\an}(M)$ for the space of real analytic functions $M\to\bR$. In general, $C^k(M,N)$ stand for spaces of maps between manifolds $M,N$.

The paper incorporates suggestions that referees have made, and a correction in the calculation of $b$ in Theorem 1.1 that Ming-Chen Xia brought to my attention.

\section{Analyticity}

In this section, after proving that analytic K\"ahler forms are dense in K\"ahler classes, we turn to our main technique of recognizing analyticity by connecting it with properties of geodesics.
The proofs of Theorems 1.1, 1.2, and 1.3 will be based on this technique.

\begin{prop}Potentials $u\in\cH$ with $\omega_u$ analytic are dense in $\cH$.
\end{prop}

\begin{proof}This will follow from Theorem 1B of H. Cartan and Grauert's embedding theorem. Consider a real 
analytic manifold $Y$ that can be analytically embedded in a complex manifold $Z$ as a totally real submanifold, and let
$\cA\to Y$ be the sheaf of germs of real analytic functions on $Y$. According to [CH, Th\'eor\`eme 1B] the sheaf cohomology 
groups $H^\nu(Y,\cA)$ vanish for $\nu=1,2,\dots$, provided the image of $Y$ in $Z$ has a Stein neighborhood basis. 
Subsequently Grauert proved in [G] that any real analytic manifold $Y$ satisfies Cartan's assumptions.

To prove the proposition, choose an open cover $\fV$ of $X$ and $p_V\in C^\infty(V)$ for $V\in\fV$ such that 
$i\partial\opartial p_V=\omega|V$.
Then $p_V-p_W$ are pluriharmonic, hence analytic on $V\cap W$. We apply H.~Cartan's vanishing theorem with
$Y$ the real analytic manifold underlying $X$. It is a simple general fact of sheaf theory that the \v Cech cohomology group
$H^1(\fV,\cA)$ of the cover $\fV$ embeds into $H^1(Y,\cA)$ (this for any sheaf of Abelian groups over any topological 
space). Hence $H^1(\fV,\cA)=0$ and there are real valued analytic $q_V$ on $V\in\fV$ such that $p_V-p_W=q_V-q_W$.
This implies that there is an $r\in C^\infty(X)$ with $q_V=p_V-r|V$.
The form $\theta=\omega-i\partial\opartial r$ is analytic, since on $V\in\fV$ it agrees with $i\partial\opartial q_V$.
Given $u\in\cH$, we can approximate $r+u$ in $C^\infty(X)$ by a sequence $r+u_k\in C^{\an}(X)$. Indeed, again by 
\cite{G}, we can assume $X$ real analytically embedded in some $\bR^m$. We extend $r+u$ to a function in 
$C^\infty(\bR^m)$, and use the fact that polynomials are dense in $C^\infty(\bR^m)|B$ for any closed ball $B\subset\bR^m$.
For large $k$ then $\omega_{u_k}=\theta+i\partial\opartial (r+u_k)>0$ are analytic, and $u_k\to u$ in $\cH$.
\end{proof}

The rest of this section expounds on the idea that analyticity properties of a K\"ahler form $\omega_u$ and of $\xi\in T_u\cH$ can be detected by considering how they relate to geodesics.
If $u\in\cH$, let $E_u\subset T_u\cH$ denote the set of tangent vectors $\xi\in T_u\cH$ for which there are $\var>0$ and a geodesic $f\colon (-\var,\var)\to\cH$ such that $f(0)=u$ and $df/dt|_{t=0}=\xi$. Further, if $\xi\in T_u\cH$, let 
$$
E_\xi=\{\eta\in T_u\cH\colon \xi+\eta,\ \xi-\eta\in E_u\}.
$$
Theorem 1.2 will be an easy consequence of the following.

\begin{thm}Let $u\in\cH$ and $\xi\in T_u\cH$. If $\omega_u,\xi$ are analytic, then $E_u, E_\xi\subset T_u\cH$ are dense.
If $\omega_u$ is not analytic, then $E_u$ is contained in a proper closed subspace of $T_u\cH$.
Similarly, if $\xi\in E_u$ is not analytic, then $E_\xi$ is contained in a proper closed subspace of $T_u\cH$.
\end{thm}

In \cite{M} Mabuchi already noted that the statement  concerning analytic $u,\xi$ follows from the Cauchy--Kovalevskaya theorem. 
Indeed, geodesics $f\colon (a,b)\to\cH$ are characterized by the equation
\begin{equation*}
{d^2f\over dt^2}={1\over 2}\Big|\text{grad } {df\over dt}\Big|^2,
\end{equation*}
where on $X$ grad and $|\quad|$ are taken with respect to $\omega_{f(t)}$.
 Thus the geodesic equation  has local solutions with $f(0)=u$, $df/dt|_{t=0}=\xi$ or $\xi\pm\eta$, if $u$, $\xi,\eta$ are analytic;   and analytic $\xi,\eta\in C^\infty(X)\approx T_u\cH$ 
are dense.---Incidentally, Rubinstein and Zelditch in \cite{RZ} prove a result that implies that  $T_u\cH\setminus E_u\subset T_u\cH$ is also dense.

What is new in Theorem 2.2 is the constraint on $E_u$, $E_\xi$ when $\omega_ u,\xi$ are not analytic. Assuming $\omega$ itself is analytic, the constraint on, say, $E_u$ is determined by the analytic wave front set of $u$.
The connection between wave front sets and geodesics is as follows.
Geodesics in $\cH$ give rise to solutions of a homogeneous complex Monge--Amp\`ere equation.
These solutions are known to have certain analyticity properties, and analytic wave front sets were designed  to detect precisely these properties.

Wave front sets are part of microlocal analysis, but we will only need the rudiments of the theory.
Here is a quick review.
If $Y$ is a finite dimensional real analytic manifold and $v$ a continuous function (or distribution or hyperfunction) on it, the analytic wave front set $\text{WF}_{\text A} v$ is a subset of $T^*Y\setminus\text{zero section}$.
Write $\pi$ for the projection $T^* Y\to Y$.
Whether $\alpha\in\text{WF}_\text{A} v$ or not depends only on the behavior of $v$ near $\pi\alpha$.
For example, $v$ is analytic in a neighborhood of $y\in Y$ if and only if $y\not\in \pi\text{WF}_\text{A}v$, see [H, Theorem 8.4.5 or Sj, Th\'eor\`eme 6.3].
Since $\text{WF}_\text{A}v$ is locally determined, it will suffice to define wave front sets when $Y\subset\bR^m$ is an open subset.
There are several equivalent ways to do this.
We will follow the one based on the FBI transformation (the acronym honoring the investigations of Fourier, Bros, and Iagolnitzer).

Write $xy=\sum x_j y_j$ for the inner product of $x=(x_j), y= (y_j)\in\bR^m$, and $x^2$ for $xx$.
Identify $T^*Y$ with $Y\times\bR^m$, and write $\alpha\in T^* Y$ as $\alpha=(\alpha',\alpha'')$ with $\alpha'=\pi\alpha\in Y$, $\alpha''\in\bR^n\approx T_{\pi\alpha} Y$.

\begin{defn}Given $v\in C(Y)$, an
$$ 
a\in T^* Y\backslash\text{ zero section }\approx Y\times (\bR^m\backslash \{0\})
$$
is not in $\text{WF}_{\text A}v$ if there are $\var, C\in (0,\infty),\ \chi\in C_0^\infty(Y)$ that is equal to 1 near $a'=\pi a$, and a neighborhood $\Omega\subset T^* Y$ of $a$ such that for $\lambda\in (1,\infty)$ and $\alpha\in\Omega$ 
\begin{equation}
\Big|\int_Y e^{i\lambda\{i(\alpha'-y)^2-\alpha'' y\}}\chi(y) v(y) dy_1\ldots dy_m\Big|\leq Ce^{-\var\lambda}.
\end{equation}
\end{defn}
It does not matter which $\chi$ we use. If (2.1) holds for some $\chi$, it holds for any $\chi$, possibly with different $\var, C$.

The integral in (2.1) is an FBI transform of $v$.
In [Sj, Chapter 6] Sj\"ostrand allows for more general phase functions $\phi$ than $\phi(y,\alpha)=i(\alpha'-y)^2-\alpha'' y$ occuring in (2.1).
In addition to analyticity near $(a',a)$ what is required is
\begin{equation}
\begin{gathered}
\text{Im}\, \phi(\alpha',\alpha)=0,\quad {\partial\phi\over\partial y}\ (\alpha',\alpha)=-\alpha'',\quad\text{ and}\\
\text{Im}\, \phi (y,\alpha)\geq\text{ const}(\alpha'-y)^2.
\end{gathered}
\end{equation}
(In Sj\"ostrand's definition the minus sign in the first line is omitted, but this is because in his version of the integral in (2.1) $\overline v$ is integrated rather than $v$. 
Also, instead of $\text{Im}\, \phi(\alpha',\alpha)=0$ he requires $  \phi(\alpha',\alpha)=0$. 
This difference is irrelevant, since from a $\phi$ satisfying (2.2) one can pass to $\phi'(y,\alpha)=\phi(y,\alpha)-\phi(\alpha',\alpha)$ without affecting $|\int_Ye^{i\lambda\phi}\chi v|$.)
Allowing more general phases is useful for understanding how  analytic changes of coordinates affect $\text{WF}_{\text A}v$---they do not, if $\text{WF}_{\text A}v$ is considered  as a subset of $T^*Y$.
We will get by by using the phase in (2.1) only, although at one point we will be obliged to estimate integrals with more general phases.
For the equivalence of various definitions of $\text{WF}_{\text A}$ see [H, Chapters VIII and IX, Sj, Th\'eor\`eme 6.5].

The main result of this section is the following.
Let $I\subset\bR$ be  an open neighborhood of 0, and $(X,\omega)$ a compact K\"ahler manifold. Denote by  
$\text{sgrad}_u$  symplectic gradient  with respect to the symplectic form $\omega_u$. If $\xi\in C^1(X)$, $\sgru_u\xi$ (also 
called Hamiltonian vector field) is the vector field that corresponds to the one--form $d\xi$ under the isomorphism 
$TX\approx T^*X$ induced by $\omega_u$.

\begin{thm}Suppose that $\omega$ is analytic.
Let $f\colon I\to\cH$ be a geodesic and
\begin{equation}
u=f(0)\in C^{\infty} ( X),\qquad \xi={df\over dt}(0)\in C^{\infty} ( X).
\end{equation}
If for some $x\in  X$ and $a\in  T^*_x  X$
\begin{equation}
a\in \wf u\cup\wf\xi,
\end{equation}
then $\langle a,\text{sgrad}_{u }\xi (x)\rangle=0$. That is, if $\xi\in E_u$ then (2.4) implies  $\langle a,\text{sgrad}_{u }\xi (x)\rangle=0$.
\end{thm}

Here $\langle \ , \ \rangle$ stands for the pairing between $T^* X$ and $TX$.
Accepting Theorem 2.4 we can complete the proof of Theorem 2.2 as follows.

\begin{proof}[Proof of Theorem 2.2]
Only the part about nonanalytic $\omega_u,\xi$  needs to be proved, and by the density of analytic K\"ahler forms, Proposition 2.1, we can assume $\omega$ is analytic. If  $u\in\cH$  and $\omega_u$ is not analytic, then $u$ is not analytic either, and there is an $a\in \text{WF}_{\text A}u\cap T^*_x X$ with some $x\in X$.
By Theorem 2.4 any $\xi\in E_u$ satisfies
$\langle a,\text{sgrad}_u\xi(x)\rangle=0$; and this constrains $E_u$ to lie in a closed hyperplane.

Next suppose that $\xi\in E_u$. We will show that if $E_\xi$ is not contained in a closed hyperplane, then $\xi$ is  analytic. Let $x\in X$ and $a\in T_xX^*\setminus\{0\}$. 
 By assumption $\langle a,\sgru_u\eta(x)\rangle \neq 0$ for some $\eta\in E_\xi$. There are two possibilities. Either   $\langle a,\sgru_u\xi(x)\rangle$ is nonzero, and by Theorem 2.4  therefore $a\notin\wf\xi$; 
or  $\langle a,\sgru_u\xi(x)\rangle = 0$, and so  $\langle a,\sgru_u\xi(x)\pm\sgru_u\eta(x)\rangle \neq 0$. 
In the second case again by Theorem 2.4  $a\notin\wf(\xi\pm\eta)$; but since $\wf\xi=\wf\big((\xi+\eta)+(\xi-\eta)\big)\subset\wf(\xi+\eta)\cup\wf(\xi-\eta)$,
this also implies $a\notin\wf\xi$. As $a$ was arbitrary, $\xi $ is indeed analytic.
\end{proof}

It is not hard to show that for generic $u\in\cH$
$$
\wf u=T^*X\backslash\text{ zero section},
$$
cf.~\cite[Theorem 8.4.14]{H}, that constructs one $u$ with prescribed wave front set.
Theorem 2.4 then implies $\xi=\text{const}$. Since according to  Rubinstein and Zelditch, or Bedford and Burns, geodesics are uniquely determined by their initial positions and velocities, see \cite[Proposition 1.1]{BB}, \cite[Theorem 2.2]{RZ}, the only geodesics through $u$ are $t\mapsto u+\const t$.
Similarly, if $u\in\cH$ is arbitrary, $\wf(\xi)=T^*X\backslash$ zero section for generic $\xi\in T_u\cH$.
Hence, again by Theorem 2.4, such $\xi$ cannot serve as the velocity vector of a geodesic.
This is of course the way it should be with ill--posed initial value problems.

The proof of Theorem 2.4 rests on the Monge--Amp\`ere interpretation of the geodesic equation and on a general result on wave front sets of functions that are analytic along the leaves of certain foliations.
We start with the latter.
In the next lemma and in its proof, $C^k(M)$,  $C(M)$ stand for complex valued functions.

\begin{lem}Let $U\subset\bC\times\bC^n$ be a neighborhood of 0, and $\cF$ a  two dimensional foliation of $U$ of class $C^1$, whose leaves are complex submanifolds (Riemann surfaces).  Assume that the leaf $L$ of $\cF$ through $0$ is transverse to the hyperplane
$$
H=\bR\times\bC^n\subset\bC\times\bC^n.
$$
If $w\in C(U\cap H)$ can be extended to a function $\tilde w\in C(U)$ that is harmonic along the leaves of $\cF$, and $b\in\wf w\cap T_0^*H$, then $b|T_0(L\cap H)=0$.
\end{lem}

\begin{proof}For brevity we will say a function is harmonic or holomorphic along $\cF$ to mean it is harmonic or holomorphic along the leaves of $\cF$.
Consider the foliation $\cF_0$ of $U$ whose leaves are $U\cap(\bC\times \{z\})$, $z\in\bC^n$.
First we deal with the special case when $\cF=\cF_0$ and $\tilde w$ is {\sl holomorphic} along $\cF_0$.
We will show that if $b\in T_0^*H$ and $b|T_0(L\cap H)\neq 0$, then $b\not\in\wf w$.

For points $y\in\bC\times\bC^n$ we use complex coordinates $y_0,y_1,\ldots,y_{n}$.
Since $T_0(L\cap H)$ is spanned by $\partial/\partial \Ree y_0$, the assumption $b|T_0(L\cap H)\neq 0$ means that the component $b''_0$ of $b=(0,b'')\in T_0^* H\approx \{0\}\times H$ is nonzero.
Define 
\begin{gather}
 \phi\colon U\times T^* (U\cap H)\approx U\times (U\cap H)\times H\to\bC,\nonumber\\
 \phi(y,\alpha)=i(\alpha'_0-y_0)^2+i\sum_1^n|\alpha'_j-y_j|^2-\alpha''_0 y_0-\text{Re}\sum_1^{n} \alpha''_j\overline y_j.
\end{gather}
In order to show $b\notin\wf w$, according to Definition 2.3 we need to produce $\var, C>0$ and $\chi\in C_0^\infty(U)$ so that $\chi\equiv 1$ in a neighborhood of 0 and
\begin{equation}
\Big|\int_{U\cap H} e^{i\lambda\phi(y,\alpha)}\chi (y) w(y) dy_0\wedge dy_{1}\wedge d\overline y_{1}\wedge\ldots\wedge dy_n\wedge d\overline y_{n}\Big|\leq C e^{-\var\lambda}
\end{equation}
for $\lambda\in (1,\infty)$ and $\alpha$ in a suitable neighborhood of $b\in\ T^*(U\cap H)$.
We will do this for continuous $\phi$ more general than (2.5).
What matters is that $\phi(\cdot,\alpha)$ is holomorphic along $\cF_0$ (i.e., as a function of $y_0$), and for $\alpha$ close to $b$ 
\begin{equation}
\text{Im}\,\phi(y,\alpha) >0\qquad\text{if }\quad \begin{cases} y\in U\cap H\setminus \{\alpha'\}\quad\text{ or }\\
\text{Im}\,y_0\text{ is small and }b''_0\text{ Im}\,y_0<0.\end{cases}
\end{equation}
This latter property is obvious for $\phi$ given in (2.5), since 
$$
\text{Im}\,\phi (y,\alpha)=-\alpha''_0 \text{ Im}\,y_0-2(\text{Im}\,y_0)^2+\sum_0^{n}|\alpha'_j-y_j|^2.
$$

To prove (2.6) we choose any $\chi\in C_0^\infty(U)$ that is 1 in a neighborhood of 0, a smooth $p\colon U\to [0,\infty)$ supported in this neighborhood such that $p(0)>0$, and with $\delta>0$ let
$$
P_\delta=\{(y_0-i\delta b''_0 p(y), y_1,\ldots,y_{n})\colon y\in U\cap H\}.
$$
If $\delta$ is small enough
\begin{multline}
\int_{U\cap H} e^{i\lambda\phi(y,\alpha)}\chi (y) w(y) dy_0\wedge dy_1\wedge d\overline y_1\wedge\ldots\wedge d\overline y_{n}=\\
 \int_{P_\delta} e^{i\lambda\phi(y,\alpha)}\chi (y) \tilde w(y) dy_0\wedge dy_1\wedge d\overline y_1\wedge\ldots\wedge d\overline y_{n},
\end{multline}
because for fixed $y_1,\ldots,y_{n}$ already the $dy_0$ integrals on both sides agree by Cauchy's theorem.
However, if $\Omega\subset T^*(U\cap H)$ is a small neighborhood of $b$ and $\delta>0$ is sufficiently small but fixed, by (2.7)
$$
\inf \{\text{Im}\,\phi(y,\alpha)\colon y\in P_\delta\cap\text{supp }\chi,\,\alpha\in\Omega\}=\var>0.
$$
Hence the right hand side of (2.8) is $O(e^{-\var\lambda})$, and (2.6) is proved in the special case.

Next we still assume $\cF=\cF_0$, but the extension $\tilde w$ to be  just harmonic along $\cF_0$.
In fact, this case is no different, because the assumption implies that $w$ even has an extension $w_1$, perhaps in a smaller neighborhood of 0, that is holomorphic along $\cF_0$.
To convince ourselves, choose $r>0$ so that
$$
\Delta=\{y\in\bC\times\bC^n\colon |y_j|\leq r,\ j=0,\ldots,n\}\subset U,
$$
and for $y=(y_j)\in\text{int}\,\Delta$ define
$$
w_1 (y)={1\over 4\pi}\int_0^{2\pi}\tilde w(re^{it},y_1,\ldots,y_{n}) \left({re^{it}+y_0\over re^{it}-y_0}+{re^{-it}+y_0\over re^{-it}-y_0} \right)\ dt,
$$
clearly holomorphic in $y_0$.
If $y\in\Delta\cap H$ then $y_0\in\bR$ and 
$$
w_1(y)={1\over 2\pi}\int_0^{2\pi}\tilde w(re^{it},y_1,\ldots,y_{n})\ \Ree\ {re^{it}+y_0\over re^{it}-y_0}\ dt=w(y)
$$
by Poisson's formula.

Lemma 2.5 in its full generality can be reduced to this special case as follows.
By applying a local biholomorphism of $(\bC^{n+1},0)$ we can arrange that the leaf $L\in\cF$ through 0 is given by $y_1=y_2=\ldots=y_{n}=0$.
Then again we need to prove that $b=(0,b'')\not\in\wf w$ if $b|T_0 (L\cap H)\neq 0$, i.e., if $b''_0\neq 0$.
We define $\phi$ by (2.5); the goal therefore is to show (2.6).

We construct a local diffeomorphism $\Theta$ of $U$ at 0 that sends (the germ of) $\cF_0$ to (the germ of) $\cF$.
Write $\pi_0$ for the projection $\pi_0(y_0,\ldots,y_{n})=y_0$.
Near $0\in U$, the fibers of $\pi_0$ are transverse to the leaves of $\cF$.
Hence, upon shrinking $U$, we can assume $\pi_0$ is injective on leaves.
For $y\in U$ let $\Theta(y)$ be the point on the leaf through $(0,y_1,\ldots,y_{n})$ for which $\pi_0\Theta(y)=y_0$. 
This $\Theta$ is indeed a $C^1$ diffeomorphism near $0$, of form
$$
\Theta (y)=(y_0,\Theta_1(y),\ldots,\Theta_n(y)),
$$
 and maps the leaves of $\cF_0$ to leaves of $\cF$. Since it fixes $L$, its differential $\Theta_*|T_0U$ is the identity.
That the leaves are complex submanifolds means the $\Theta_j(y)$ depend holomorphically on $y_0$.
This implies that $\partial\Theta_j(y)/\partial y_k$ and  $\partial\Theta_j(y)/\partial \overline y_k$ also depend holomorphically on $y_0$. Hence, modulo terms that are multiples of $dy_0\wedge d\overline y_0$,
$$
dy_0\wedge d\Theta_1\wedge d\overline\Theta_1\wedge\ldots\wedge d\Theta_n\wedge d\overline\Theta_n=J dy_0\wedge dy_{1}\wedge d\overline y_{1}\wedge\ldots\wedge dy_{n}\wedge d\overline y_{n},
$$
where $J=\sum_k A_k\overline B_k$ with $A_k, B_k\in C(U)$ that depend holomorphically on $y_0$.
To estimate the integral in (2.6) we 
substitute $\Theta(y)$ for $y$---which is legitimate if $\chi$ is supported in a sufficiently small neighborhood of 0. This transforms the integral in (2.6) into
\begin{equation}
\int_{U\cap H} e^{i\lambda\phi(\Theta(y),\alpha)}\chi(\Theta(y)) w(\Theta(y))J (y) dy_0\wedge dy_1\wedge\ldots\wedge dy_{n}\wedge d\overline{y}_{n}.
\end{equation}

Now we are in a situation already dealt with. We define a new phase function $\phi'$ by letting 
$(y_0,\ldots,y_n)^\dag=(\overline y_0,y_1,\ldots,y_n)$ and 
\begin{align*}
\phi'(y,\Theta^*\alpha)&=i(\alpha_0'-y_0)^2+i\sum_1^n(\alpha_j'-\Theta_j(y))(\overline{\alpha_j'}-\overline{\Theta_j(y^\dag)})\\
&\phantom{a+b+c+d}
-\alpha''y_0-\sum_1^n(\alpha''_j\overline{\Theta_j(y^\dag)}+\overline{\alpha_j''}\Theta_j(y))/2\\
&=i(\alpha_0'-y_0)^2-\alpha''y_0+\sum_1^n\big(i|\alpha_j'-\Theta_j(y)|^2-\Ree\alpha''_j\overline{\Theta_j(y)}\big)+R\\
&=\phi\big(\Theta(y),\alpha\big)+R,
\end{align*}
where the error term
$$
R=\sum_1^n\big(i\alpha_j'-i\Theta_j(y)+\alpha_j''/2\big)\big(\overline{\Theta_j(y)}-\overline{\Theta_j(y^\dag)}\big).
$$
In particular, $R=0$ when $y\in H$, and the phase  $\phi(\Theta(y),\alpha)$ in (2.9) can be replaced by 
$\phi'(y,\Theta^*\alpha)$.
As $\Theta_j(0)=\partial\Theta_j/\partial y_0(0)=0$ for $j\ge 1$, we can estimate 
$$
\Theta_j(y)-\Theta_j(y^\dag)=O(|y||\text{Im}\,y_0|)\quad\text{and}\quad
|R|= O(|\alpha|+|y|)|y||\text{Im}\,y_0|,
$$
as $y\to 0$. Since $\pi{\Theta^*}^{-1}\beta=\Theta(\pi\beta)$,
$$
\text{Im}\,\phi'(y,\beta)=-\alpha''_0 \text{ Im}\,y_0-2(\text{Im}\,y_0)^2+\sum_0^{n}|\alpha'_j-\Theta_j(y)|^2+\text{Im}\, R,\quad\beta=\Theta^*\alpha,
$$
satisfies for $y$ close to $0$ and  $\beta$ close to $\Theta^*b=b$
$$
\text{Im}\,\phi'(y,\beta)>0\quad{\text{if}}\quad y\in H\setminus\{\beta'\}\quad\text{or}\quad b_0''\,\text{Im}\,y_0<0,
$$
cf. (2.7). Clearly, $\phi'(y,\beta)$ is holomorphic in $y_0$.
The function $w\circ\Theta$, in turn, has a continuous extension to a neighborhood of $0\in U$ that is harmonic along $\cF_0$, hence as we saw, also an extension that is holomorphic along $\cF_0$.
This implies that
$$
w(\Theta(y)) J(y)=w(\Theta(y))\sum_k A_k(y) \overline{B_k(y^\dag)}, \qquad y\in U\cap H,
$$
also has a continuous extension, holomorphic along $\cF_0$.
By applying to (2.9), with the phase written as $\phi'(y,\Theta^*\alpha)$, the estimate we have already proved when $\cF=\cF_0$, it follows that (2.6) holds for general $\cF$, and the proof is complete.
\end{proof}
\begin{proof}[Proof of Theorem 2.4.]
Let $S=\{s\in\bC\colon\text{ Im}\,s\in I\}$ and pull back $\omega$ along the projection $S\times X\to X$ to a $(1,1)$--form $\Omega$. Note that $f$, as any geodesic, is $C^{\infty}$ (see \cite[Lemma 6.2]{L2}), and so  $v$ defined by $v(s,\cdot)=f(\text{Im}\,s)$ is in $C^{\infty}(S\times X)$.
According to Semmes, the geodesy of $f$ means
$$
(\Omega+i\partial\overline\partial v)^{n+1}=0\qquad\text{ on }S\times X.
$$
We will obtain the conclusion of Theorem 2.4 by studying this Monge--Amp\`ere equation.
Our arguments will be local, and for this reason no generality is lost if we take $X$ to be an open subset of $\bC^n$, and $\omega$ to have an analytic potential $p$, so that $\omega=i\partial\overline\partial p$.
Also, $x\in  X$ of the theorem can be taken to be $0\in\bC^n$.
The function
\begin{equation}
\tilde w(s,x)=v(s,x)+p(x)
\end{equation}
then solves
\begin{equation}
(\partial\opartial\tilde w)^{n+1}=0.
\end{equation}
Since $\tilde w(s,\cdot)$ is strongly plurisubharmonic, from (2.11) we conclude that Ker $\partial\opartial\tilde w$ defines a smooth two dimensional subbundle of $T(S\times X)$.

According to Bedford and Kalka  \cite[Theorem 2.4]{BK}  $S\times X$ is smoothly foliated by Riemann surfaces tangent to $\Ker\partial\bar\partial\tilde w$. Call this foliation $\cF$.
Since $\tilde w(s,\cdot)$ is strongly plurisubharmonic, the leaves are transverse  to $H=\bR\times\bC^n$; 
$\tilde w$ is harmonic, $\partial\tilde w/\partial s$ is holomorphic along $\cF$.
(Other components of $\partial\tilde w$ are also holomorphic along $\cF$, but this we will not need in the proof.) 

We apply Lemma 2.5, but to do so we have to consider the leaf $L$ of $\cF$ through $0\in H$ and compute $T_0(L\cap H)$.
We claim that $T_0(L\cap H)$ is spanned by the vector 
$$
4\Ree\partial/\partial s-\sgru_u\xi (0)\in T_0(\bC\times X).
$$
The vector is clearly tangential to $H$, but we need to check it is tangential to $L$, i.e.~to $\Ker\partial\opartial\tilde w|T_0 (S\times X)$. This is essentially known, see \cite[p. 23]{Do1}, but we do the calculation anyway.

To simplify, we choose coordinates $x_j$ on $X$ so that $\omega_u=i\sum d x_j\wedge d\overline x_j$ at $0\in X$.
In view of (2.3), (2.10) this means that at $0\in\bC\times\bC^n$
\begin{eqnarray*}
\partial\opartial\tilde w&=&\tilde w_{s\overline s} ds\wedge d\overline s+\sum_j\tilde w_{s\overline x_j} ds\wedge d\overline x_j+\sum_j\tilde w_{\overline s x_j} dx_j\wedge d\overline s+\sum_j dx_j\wedge d\overline x_j\\
&=&\tilde w_{s\overline s} ds\wedge d\overline s-{i\over 2}\ \sum_j \xi_{\overline x_j} ds\wedge d\overline x_j+{i\over 2}\sum_j\xi_{x_j} dx_j\wedge d\overline s+\sum_j dx_j\wedge d\overline x_j.
\end{eqnarray*}
For a vector Re $(\lambda_0\partial/\partial s+\sum\lambda_j \partial/\partial x_j)$ to be in $\Ker \partial\opartial\tilde w|T_0(S\times X)$, its contraction with $\partial\opartial \tilde w$ should vanish.
Modulo multiples of $ds,d\overline s$, the contraction is
$$
i\,\text{Im}\,\big(\sum\lambda_j d\overline x_j-\frac i2\sum\lambda_0 \xi_{\overline x_j} d\overline x_j\big).
$$
If this vanishes, $\lambda_j=i\lambda_0\xi_{\bar x_j}/2$. Hence $\Ker\partial\opartial\tilde w|T_0 (S\times X)$, known to be two dimensional, consists of vectors of form Re $\lambda_0(\partial/\partial s+(i/2)\sum\xi_{\overline x_j}\partial/\partial x_j)$.
For this vector to be tangent to $H$, $\lambda_0$ must be real, i.e., $T_0(L\cap H)$ is indeed spanned by
$$
\Ree\big(4\frac{\partial}{\partial s}+ 2i \sum \xi_{\overline x_j}(0)\frac \partial{\partial x_j}\big)=4\Ree\frac\partial{\partial s}-\sgru_u\xi (0).
$$

Now $\tilde w|H$ is the pull back of $v(0,\cdot)+p=u+p$ by the projection $\bR\times  X\to  X$, and $2i\partial\tilde w/\partial s|H=\partial\tilde w/\partial\text{Im}\,s|H$ is the pull back of $\xi$. Since $p$ is analytic, $\wf u=\wf(u+p)$.
It is easy to read off from (2.1) that pull back along a projection commutes with taking wave front sets. Hence
 $a\in (\wf u\cup \wf\xi)\cap T_0^* (X)$ pulls back to some
$$
b\in \big(\wf (\tilde w|H)\cup \wf (\partial \tilde w/\partial s|H)\big)\cap T_0^* (\bR \times X).
$$
Lemma 2.5 implies $b|T_0(L\cap H)=0$, and this is equivalent to $\langle a,\sgrad\xi(0)\rangle=0$, as claimed.
\end{proof}

\section{Isometries and analyticity}

In this section we will prove  Theorem 1.2 as well as a generalization. Theorem 1.2  is an easy consequencence of  Theorem 2.2 and 
\begin{lem}If $\cU\subset\cH$ is open, any $C^1$ isometry $F\colon \cU\to\cH'$ maps geodesics to geodesics.\end{lem}
\begin{proof}This depends on the characterization of geodesics $[a,b]\to\cH$ as critical points of the energy
$$
\cE(f)=\frac 12\int_a^b|df(t)/dt|^2_{f(t)}\, dt,\qquad f\in C^1\big([a,b],\cH\big).
$$
On the one hand, it follows from \cite[p. 219, Corollary 3]{Ch} that geodesics minimize $\cE(f)$ among $C^1$ curves $f\colon[a,b]\to\cH$ with fixed $f(a), f(b)$. 
(Chen talks about minimizing length instead of energy, but one can be reduced to the other by means of the Cauchy--Schwarz inequality.)  
On the other hand, by \cite[Corollary 6.4]{L2}, quite generally $C^1$ isometries map energy minimizing curves to geodesics, under a mild condition that Mabuchi's metric meets. 
These two facts together prove the lemma.
\end{proof}
\begin{proof}[Proof of Theorem 1.2.]Recall the definitions of the sets $E_u,E_\xi\subset T_u\cH$, $u\in\cH,\,\xi\in T_u\cH$, 
given in the paragraph preceding Theorem 2.2.  Lemma 3.1 implies that an  isometry 
$F$ maps $E_u$ to  $E_{F(u)}$ and $E_\xi$ to $E_{F_*\xi}$.  Suppose $\omega_u$ is analytic.
 By Theorem 2.2, $E_u\subset T_u\cU$ is dense, whence so is $F_*E_u=E_{F(u)}\subset T_{F(u)}\cH'$, and another application of Theorem 2.2 gives that $\omega_{F(u)}$ is analytic.

Next suppose that $\xi\in T_u\cU$ is also analytic. Again, by Theorem 2.2  $E_\xi$, therefore $F_*E_\xi=E_{F_*\xi}$ are dense, which implies $F_*\xi$ is analytic.
\end{proof}

Henceforward in this section we assume that $\omega,\omega'$ are analytic. We generalize Theorem 1.2 to certain maps into $\cH$.
\begin{defn}Let $P$ be a real analytic manifold.
We say that a map $u\colon P\to\cH$ is very analytic if $v\colon P\times X\to\bR$ defined by $v(\tau,\cdot)=u(\tau)$ is analytic.
If $u$ is such, we say a section $\xi$ of $u^*T\cH$, i.e., a map
$$
P\ni\tau\mapsto\xi(\tau)\in T_{u(\tau)}\cH
$$
is very analytic if $\eta\colon P\times X\to\bR$ defined by $\eta(\tau,\cdot)=\xi(\tau)\in T_{u(\tau)}\cH\approx C^\infty(X)$ is also analytic.
\end{defn}
\begin{thm}Let $\cU\subset\cH$ be open, $F\colon\cU\to\cH'$ a $C^\infty$ isometry, and  $P$  a real analytic manifold.
If  $u\colon P \to\cU$ is very analytic and  $\xi$ is a very analytic section of $u^*T\cH$, then $u'=F\circ u$ and $\xi'= F_*\xi$ are very analytic.
\end{thm}

\begin{proof} We will reduce Theorem 3.3 to Theorem 1.2 as follows. 
First assume that $P=\bR^m/\bZ^m$ is a torus, and consider the complex torus $\bT=\bC^m/(\bZ+i\bZ)^m=P\oplus iP$, endowed with a K\"ahler form $\omega_\bT$. 
On $\bT\times X$ we have a K\"ahler form $\omega_\bT\oplus\omega$ and the corresponding space $\cH(\omega_\bT\oplus\omega)$ of K\"ahler  potentials; and similarly, the space  $\cH(\omega_\bT\oplus\omega')$ of potentials on $\bT\times X'$. 
Our $F$ induces a $C^\infty$ map $\hat F$ of
$$
\hat\cU=\{v\in  \cH(\omega_\bT\oplus\omega)\colon v(t,\cdot)\in\cU\text{ for all }t\in\bT\}
$$
into $C^\infty(\bT\times X')$, 
$$
\hat F(v)(t,\cdot)=F\big( v(t,\cdot)\big).
$$
We let $\cV=\hat F^{-1} \cH(\omega_\bT\oplus\omega')$. 
The restriction $G=\hat F|\cV$ is  a diffeomorphism on its image, and its differential acts on $\eta\in T_v  \cH(\omega_\bT\oplus\omega)$ by
$$
(G_*\eta)(t,\cdot)=F_*\big(\eta(t,\cdot)\big).
$$
This implies that $G$ is an isometry. Hence by Theorem 1.2 it maps analytic $v$, $\eta$ to analytic $G(v)$, $G_*\eta$.

Take $u,\xi$ as in the theorem. Denoting by $\Ree$ the projection $\bT=P\oplus iP\to P$, define $v,\eta\in C^\infty(\bT\times X)$ by
$$
v(t,\cdot)=u(\Ree t),\qquad\eta(t,\cdot)=\xi(\Ree t).
$$
If the K\"ahler form $\omega_\bT$ is chosen sufficiently large, $v\in  \cH(\omega_\bT\oplus\omega)$ and $\hat F(v)\in  \cH(\omega_\bT\oplus\omega')$, that is, $v\in\cV$. 
Since $v,\eta$ are analytic, so are $G(v),G_*\eta$, which is the same thing as saying that $u',\xi'$ are very analytic.

Next consider $P=(-2,2)^m\subset\bR^m$. For $t\in\bR^m$ let $p(t)=(\sin 2\pi t_1,\ldots,\sin 2\pi t_m)\in P$. The maps
$$
u\circ p\colon\bR^m\to\cH,\qquad \xi\circ p\colon \bR^m\to  C^\infty(X)\approx T_{u(p(t))}\cH
$$
descend to very analytic maps $\check u\colon\bR^m/\bZ^m\to\cH$,  $\check\xi\colon \bR^m/\bZ^m\to C^\infty(X)$. It follows that $F\circ\check u$ and $F_*\check\xi$ are very analytic, and so are $F\circ u$ and $F_*\xi$ over $(-1/4,1/4)^m$. Hence the case of general $P$ follows, as Theorem 3.3 is of local nature.
\end{proof}

It would be of some interest to clarify whether Theorem 3.3 holds for $C^1$ isometries as well. If it did, then so would the uniqueness theorem, Theorem 1.3.
By different means we could only prove that $C^1$ isometries send very analytic maps to $C^\infty$ maps, something that is not quite as useful.

\section{The proof of Theorem 1.1}

We consider K\"ahler manifolds $(X,\omega),(X',\omega')$, their spaces of potentials $\cH,\cH'$, an open $\cU\subset\cH$ and a $C^1$ isometry
$F\colon\cU\to\cH'$.
We denote by $\{\ , \ \}$, $\{\ , \ \}'$ the Poisson brackets on $T_u\cU\approx C^\infty (X)$, $T_{F(u)} \cH'\approx C^\infty (X')$, induced by
$\omega_u$, $\omega_{F(u)}$.
Finally, we write $\Phi$ for $F_*|T_u\cU$, viewed as a map $C^\infty(X)\to C^\infty(X')$.

\begin{lem}
For any $\xi,\eta,\zeta\in C^\infty(X)$ 
\begin{equation}
\{\{\Phi\xi,\Phi\eta\}', \Phi\zeta\}'=\Phi \{\{\xi,\eta\},\zeta\}.
\end{equation}
\end{lem}

As we will see, the proof would be very quick if we had assumed that $F$ was $C^2$. To prove for $C^1$  isometries, we need a substitute for the exponential map: an exponential surface.
\begin{defn} An exponential surface is a $C^1$ map $e\colon\Delta\to\cH$ of some disc $\Delta\subset\bR^2$ centered at the origin, whose restrictions to radii of $\Delta$ are unit speed geodesics. 
\end{defn}
In particular,  $e_*$ is isometric between $T_0\Delta$, with the Euclidean metric, and  its image, with the metric inherited from $\cH$.
\begin{lem}If $\omega_u$ is analytic and $L\colon T_0\bR^2\to T_u\cH$ is a linear map that is an isometry on a plane $P$ consinsting of analytic $\xi\in  T_u\cH$, then there is a disc $\Delta\subset\bR^2$ and a unique exponential surface $e\colon\Delta\to\cH$ such that $e_*|T_0\Delta=L$. 
This $e$ is very analytic.
\end{lem}
\begin{proof}Uniqueness follows because a geodesic is uniquely determined by its initial position and velocity (\cite[Proposition 1.1]{BB}, \cite[Theorem 1.1]{RZ}). 
As to existence and analyticity, write $\sigma,\tau$ for the coordinates on $\bR^2$, and let $\xi=L(\partial/\partial\sigma)$, $\eta=L(\partial/\partial\tau)$, orthonormal vectors. We can assume that $\omega$ is analytic; then $u$ will be analytic, too.
By  the Cauchy--Kovalevskaya theorem on a neighborhood of $\{0\}\times X\subset\bR^3\times X$ we can solve the equation
\begin{align*}
\frac{\partial^2v(t,\sigma,\tau,x)}{\partial t^2}&=\frac 12\Big|\grad\frac{\partial v(t,\sigma,\tau,x)}{\partial t}\Big|^2\\
v(0,\sigma,\tau,\cdot)&=u,\qquad \frac{\partial v}{\partial t}(0,\sigma,\tau,\cdot)=\sigma\xi+\tau\eta,
\end{align*}
where $\text{grad}$, taken on $X$, and length $|\quad|$ are with respect to the metric determined by $\omega+i\partial\bar\partial v(t,\sigma,\tau,\cdot)$.
The meaning of these equations is that for fixed $\sigma,\tau$ the curve $t\mapsto v(t,\sigma,\tau,\cdot)\in\cH$ is a geodesic through $u$, of speed $\sqrt{\sigma^2+\tau^2}$. 
Any solution can be reparametrized to another solution $v_\lambda(t,\sigma,\tau,x)=v(t/\lambda,\lambda\sigma,\lambda\tau,x)$, $\lambda\in\bR\setminus\{0\}$. Since the germ of an analytic solution is unique, $v_\lambda=v$, or
$$
v(\lambda t,\sigma,\tau,x)=v(t,\lambda\sigma,\lambda\tau,x).
$$
In particular, $v$ is analytic on $(-2,2)\times\Delta\times X$ if $\Delta$ is a sufficiently small disc centered at $0$. We let $e(\sigma,\tau)=v(1,\sigma,\tau,\cdot)\in\cH$ for $(\sigma,\tau)\in\Delta$. 
Clearly $e$ is very analytic, and its restriction to a radius $t\mapsto(\sigma_0t,\tau_0t)$ of $\Delta$ (where $\sigma_0^2+\tau_0^2=1$) is
$$
t\mapsto e(\sigma_0t,\tau_0t)=v(t,\sigma_0,\tau_0,\cdot),
$$
a unit speed geodesic. Therefore $e$ is the exponential surface sought.
\end{proof}

\begin{proof}[Proof of Lemma 4.1]By \cite[Theorem 4.3]{M}, the curvature $R$ of the Mabuchi metric acts on $T_u\cH$ by
$$
R(u,\xi,\eta)\zeta=-\{\{\xi,\eta\},\zeta\}/4,\qquad \xi,\eta,\zeta\in C^\infty (X)\approx T_u\cH,
$$
and similarly for the curvature $R'$ of $\cH'$. (Mabuchi's formula does not contain the factor $-1/4$, due to differing
 conventions).
Therefore what we need to show is that $R$ and $R'$ correspond under isometries. This is rather obvious for $C^2$ isometries, see \cite[Lemma 5.1]{L2}, but not so for the $C^1$ isometries we are studying here. 
Since curvature at $u\in\cH$ is determined by sectional curvatures $K(P)$ along planes $P\subset T_u\cH$, proving
\begin{equation}
K'(F_*P)=K(P),
\end{equation}
where $K'$ denotes sectional curvature in $\cH'$, would suffice.

Sectional curvature can be computed from exponential surfaces.
Suppose an exponential surface $e$ is even $C^\infty$. Let  $e_*T_0\Delta=P\subset T_u\cH$ and $L_r$ be the length of the circle $[0,2\pi]\ni\theta\mapsto e(r\cos\theta,r\sin\theta)$. Then
\begin{equation}
L_r=2\pi r\big(1-K(P)r^2/6+o(r^2)\big),\qquad\text{as }r\to 0,
\end{equation}
see \cite[Lemma 7.1]{L2}.

We prove (4.2) first when $\omega_u$ is analytic and $P$ is spanned by analytic $\xi,\eta$.
By Lemma 4.3 there is a very analytic exponential surface $e\colon \Delta\to\cU$ such that $P=e_*T_0\Delta$.
By Lemma 3.1 $F$ maps geodesics to geodesics, with the same speed; hence $e'=F\circ e\colon\Delta\to\cH'$ is also an exponential surface. 
Finally, by Theorem 1.2 $u'=F(u) $ and $\xi'=F_*\xi$, $\eta'=F_*\eta$ are analytic. It follows, again by Lemma 4.3, that $e'$ is very analytic near the origin. 
Since the circles
$$
\theta\mapsto e(r\cos \theta,r\sin \theta)\quad\text{and}\quad \theta\mapsto e'(r\cos \theta,r\sin \theta),\qquad 0\le\theta\le 2\pi,
$$
have the same lengths $L_r=L'_r$,  (4.3) implies $K'(F_*P)=K(P)$.

Once we know this when $u,\xi,\eta$ are analytic, the general case follows by density, since $K(\xi\wedge\eta)$ depends continuously on $\xi,\eta$.
\end{proof}

For the rest of this section, until the last two paragraphs, we will work with arbitrary connected, compact, smooth symplectic manifolds $(X,\omega),(X',\omega')$ that we now fix, and
the induced Poisson brackets $\{\ , \ \}$, $\{\ , \ \}'$ on $C^\infty(X)$, $C^\infty(X')$.
We will obtain the general form of linear isomorphisms $\Phi\colon C^\infty(X)\to C^\infty (X')$ that satisfy (4.1), eventually also using that $\Phi$
preserves the $L^2$ norm.

\begin{lem}Suppose a linear isomorphism
$\Phi\colon C^\infty (X)\to C^\infty (X')$ satisfies (4.1) for\newline
all $\xi,\eta,\zeta\in C^\infty (X)$.
Then either $\{\Phi\xi,\Phi\eta\}'=\Phi\{\xi,\eta\}$ for all $\xi,\eta\in C^\infty(X)$ or $\{\Phi\xi,\Phi\eta\}'=-\Phi\{\xi,\eta\}$ for all $\xi,\eta\in C^\infty(X)$.
\end{lem}

A vector space $V$ endowed with a trilinear map $V\times V\times V\to V$ satisfying certain axioms is called a Lie triple system.
An example is $C^\infty(X)=V$ endowed with the triple bracket $\{\{\ , \ \},\}$.
Lemma 4.4 says that an isomorphism of these Lie triple systems is either an isomorphism or an anti--isomorphism of the corresponding  Lie algebras.
A related finite dimensional result,  that we are going to use, is due to E. Cartan.

\begin{thm}
Let $(\fg,[\ ,])$ and $(\fg',[\ ,]')$ be finite dimensional Lie algebras over $\bR$ and $f\colon\fg\to\fg'$ a linear isomorphism that satisfies
$$
[[f(a),f(b)]',f(c)]'=f[[a,b],c]\quad\text{ for }a,b,c\in\fg.
$$
If $\fg$ is simple, then either $[f(a),f(b)]'=f[a,b]$ for all $a,b\in\fg$ or $[f(a),f(b)]' =-f[a,b]$ for all $a,b\in\fg$.
\end{thm}

In \cite{CE} Cartan does not provide a proof but gives a hint what tools to use.
In the Appendix we will write out a proof that, sure enough, uses the Cartan--Killing theory of semisimple Lie algebras.

To prove Lemma 4.4, in addition to Cartan's theorem we will also need

\begin{lem}The commutator algebra $\{C^\infty(X),C^\infty(X)\}$ is
\begin{equation}
N=\{\xi\in C^\infty(X)\colon\int_X\xi\omega^n=0\},
\end{equation}
where $\dim_{\bR}X=2n$.
More precisely, suppose $(\zeta_1,\ldots,\zeta_m)\colon X\to\bR^m$ is a smooth embedding.
Then $\xi\in C^\infty(X)$ is in $N$ if and only if it can be represented $\xi=\sum_1^m \{\eta_j,\zeta_j\}$ with suitable $\eta_j\in C^\infty(X)$.
\end{lem}

This fact has been discovered and rediscovered in different contexts.
Possibly the earliest reference is \cite{AG} by Atkin and Grabowski, where the result is shown if not explicitly formulated, see the proofs of (5.2) Theorem and (2.6) Proposition.

\begin{lem}For any $p\in X$
$$
I_p= \{\xi\in C^\infty(X)\colon\xi-\xi(p)\text{ vanishes to infinite order at }p\}
$$
is a maximal ideal of the Lie algebra $(C^\infty(X), \{\ , \ \})$.
These and $N$ are all the maximal ideals.
\end{lem}

Further down, if $q\in X'$, we will write $I'_q\subset C^\infty(X')$ for the ideal analogously defined.

\begin{proof}To prove Lemma 4.7 it will suffice  to show that if $I\subset C^\infty(X)$ is an ideal, then either $I\subset I_p$ for some $p\in X$ or $I\supset N$ (note that codim $N=1$).

Suppose $I\not\subset I_p$ for any $p$, so that for every $p\in X$ there is a $\xi\in I$ such that $\xi-\xi(p)$ vanishes to a finite order only.
If this order is $>1$, with a suitable $\alpha\in C^\infty(X)$ and
$$
\xi'=(\sgrad \alpha)(\xi-\xi(p))=\{\alpha,\xi\}\in I,
$$
$\xi'=\xi'-\xi'(p)$ will vanish to one order less.
Applying this repeatedly we obtain $\xi''\in I$ such that $\sgrad\xi''\neq 0$ at $p$.
Given a nonzero tangent vector $t\in T_pX$, with suitable $\beta\in C^\infty(X)$ 
$$
\eta=(\sgrad \xi'')\beta=\{\xi'',\beta\}\in I
$$
has nonzero $t$--derivative.
We can choose finitely many such $\eta_1,\ldots,\eta_k\in I$ so that the map $(\eta_1,\ldots,\eta_k)\colon X\to\bR^k$ is an immersion.
Furthermore, given distinct $p,q\in X$, with a suitable $j=1,\ldots,k$ and $\gamma\in C^\infty(X)$
$$
\zeta=(\sgrad\eta_j)\gamma=\{\eta_j,\gamma\}\in I
$$
will vanish at $p$ but not at $q$.
Adjoining finitely many such $\zeta$'s to the $\eta$'s we obtain $\zeta_1,\ldots,\zeta_m\in I$ that embed $X$ in $\bR^m$.
But then Lemma 4.6 implies $N\subset \{C^\infty (X),I\}\subset I$, as needed.
\end{proof}

\begin{proof}[Proof of Lemma 4.4]
It follows from Lemma 4.6 that various Lie theoretic notions in $(C^\infty(X),\{\ \})$ and in $(C^\infty (X'),\{ \ \}')$ can be explained in terms of the triple brackets, and so they must correspond under $\Phi$.
For example, $\zeta\in C^\infty(X)$ is in the center (i.e., is constant) if and only if $\sgrad\zeta$ annihilates $C^\infty(X)$.
But this is the same as annihilating $\{C^\infty(X), C^\infty(X)\}$.
The upshot is that $\zeta$ is in the center if and only if $\{\{\xi,\eta\},\zeta\}=0$ for all $\xi,\eta\in C^\infty(X)$.
Hence $\Phi$ maps constants to constants.
Similarly, a subspace $I\subset C^\infty (X)$ is a Lie ideal if and only if $\{\{C^\infty(X),C^\infty(X)\},I\}\subset I$.
Hence ideals, and so maximal ideals too, correspond under $\Phi$.
Our subsequent analysis was inspired by Omori's treatment of isomorphisms of Lie algebras of symplectic vector fields (who, in turn, was inspired by Pursell--Shanks and their referee, forever anonymous now), see \cite{O,PS}.

Since the infinite codimensional maximal ideals in $C^\infty(X)$, $C^\infty(X')$ correspond under $\Phi$, by Lemma 4.7 there is a bijection $\psi\colon X\to X'$ such that $\Phi(I_p)=I'_{\psi(p)}$.
Consider for $p\in X$ the subalgebra $A_p\subset C^\infty(X)$ and the ideal $B_p\subset A_p$
$$
A_p=\{\xi\colon\sgrad \xi(p)=0\},\qquad B_p=\{\xi\colon\sgrad\xi\text{ vanishes at }p\text{ to order }\ge 2\},
$$
and similarly $B'_q\subset A'_q\subset C^\infty(X')$.
It is immediate that if $\xi\not\in A_p$ then any $\zeta\in C^\infty(X)$ can be written near $p$ as $\zeta=(\sgrad\xi)\eta$ with some $\eta\in \{C^\infty(X),C^\infty(X)\}=N$.
In particular, $\{\{C^\infty(X),C^\infty(X)\},\xi\}+I_p=C^\infty(X)$.
This leads to the characterization
\begin{gather*}
A_p=\{\xi\colon \{\{C^\infty(X),C^\infty(X)\},\xi\}+I_p\neq C^\infty(X)\},\text{ and similarly}\\
B_p=\{\xi\colon \{\{C^\infty(X),C^\infty(X)\},\xi\}\subset A_p\}.
\end{gather*} 
It follows that $\Phi(A_p)=A'_{\psi(p)}$ and $\Phi(B_p)=B'_{\psi(p)}$.
Hence $\Phi$ descends to a linear isomorphism of the quotient algebras
$$
f\colon A_p/B_p=\fg\to A'_{\psi(p)}/B'_{\psi(p)}=\fg',
$$
that is compatible with the inherited triple brackets.

It is not hard to see that $\fg$ is isomorphic to the symplectic Lie algebra sp$(2n,\bR)$.
Indeed, if we introduce local coordinates $x_1,\ldots,x_{2n}$ in $X$ centered at $p$ so that $\omega=\sum_1^n dx_i\wedge dx_{i+n}$, and we associate with $\xi\in A_p$ the quadratic part of its Taylor series at $p$, this map descends to a linear isomorphism between $A_p/B_p$ and the space $Q$ of 2--homogeneous polynomials in $x_1,\ldots,x_{2n}$.
The isomorphism respects the Poisson bracket, and so $\fg\approx (Q,\{\ , \})\approx \text{sp} (2n,\bR)$.

We apply Cartan's theorem above to conclude that for $\xi,\eta\in A_p$
$$
\{ \Phi\xi,\Phi\eta\}'=\var \Phi \{ \xi,\eta\}\qquad \text{mod } B'_{\psi(p)},
$$
with $\var=\var_p=\pm 1$ independent of $\xi,\eta$.
This relation can be extended as follows.
Suppose first $\xi\in A_p$ but $\eta\in C^\infty(X)$ is arbitrary.
Since $\fg=\text{sp} (2n,\bR)$ is simple, there are  $\xi_1,\xi_2\in A_p$ such that $\xi=\{\xi_1,\xi_2\}$ mod $B_p$.
Then
$$
\{\Phi\xi,\Phi\eta\}'=\var \{\{\Phi\xi_1,\Phi\xi_2\}' , \Phi\eta\}'=\var\Phi\{\{\xi_1,\xi_2\},\eta\}=\var\Phi\{\xi,\eta\}\qquad\text{mod }A'_{\psi(p)}.
$$
If now $\xi,\eta\in C^\infty(X)$ are arbitrary, there are $\alpha\in A_p$, $\beta\in C^\infty(X)$ such that $\{\alpha,\beta\}=\xi-\xi(p)$ near $p$, which gives at $\psi(p)$ 
$$
\{\Phi\xi,\Phi\eta\}'=\var\{\{\Phi\alpha,\Phi\beta\}',\Phi\eta\}'=\var\Phi\{\{\alpha,\beta\},\eta\}=\var\Phi\{\xi,\eta\}.
$$
Writing $\psi(p)=q$, this means 
$$
\{\Phi\xi,\Phi\eta\}' (q)=\var_{\psi^{-1}(q)}(\Phi\{\xi,\eta\})(q).
$$
Since in this equation all quantities except $\var_{\psi^{-1}(q)}$ obviously depend continuously on $q$, it follows that $\var_p=1$ for all $p$ or $\var_p=-1$ for all $p$, which proves the lemma.
\end{proof}

\begin{proof}[Proof of Theorem 1.1]
By continuing with our analysis of $\Phi=F_*|T_u\cU$ in the same spirit, it would not take too much to find $\varphi$ (it would be $\psi^{-1}$) and $a,b$ of the theorem.
However, instead we can conveniently refer to a theorem of Atkin and Grabowski concerning isomorphisms of Poisson--Lie algebras of symplectic manifolds.
For this we only need $(X,\omega),(X',\omega')$ to be connected compact smooth symplectic manifolds; then any Lie algebra isomorphism $\Phi\colon C^\infty(X)\to C^\infty(X')$ determines a diffeomorphism $\varphi\colon X'\to X$ and numbers $a,b$ so that
\begin{equation}
a\varphi^*\omega=\omega'\qquad\text{ and }\qquad\Phi\xi=a\varphi^*\xi-b\int_X\xi\omega^n
\end{equation}
$(\dim_{\bR}X=\dim_{\bR} X'=2n)$.
\cite[(8.10) Theorem]{AG} is formulated differently, among other things because it covers disconnected manifolds as well; but for connected $X,X'$ it boils down to (4.5), if one takes into account that $N$ is the commutator algebra of $C^\infty(X)$ and the center consists of constants.
To identify $a,b$, let $V=\int_X\omega^n$, $V'=\int_{X'}\omega'^n$. Thus
\begin{equation*}
V'=a^n\int_{X'}\varphi^*\omega^n=\sigma a^n\int_X\omega^n=\sigma a^n V,
\end{equation*}
where $\sigma=\pm 1$ depending on whether $\varphi$ preserves orientation or not.
If $\Phi$ also preserves $L^2$ norms, then 
\begin{gather*}
V'\int_X\xi^2\omega^n=V\int_{X'} (\Phi\xi)^2 \omega'^n=V a^n\int_{X'}\Big(a\varphi^*\xi-b\int_X\xi\omega^n\Big)^2(\varphi^*\omega)^n, \qquad\text{i.e.},\\
\int_X\xi^2\omega^n=a^2\int_X \xi^2\omega^n+ b(bV-2a)\Big(\int_X\xi\omega^n\Big)^2,\qquad\xi\in C^\infty(X).
\end{gather*}
When $\int_X\xi\omega^n=0$, we obtain $a^{2}=1$, so $a=\pm 1$, whence $b=0\text{ or } 2a/V$ follows.
This takes care of the form of $\Phi=F_*|T_u\cU$, assuming it preserves Poisson brackets.
If it does not, then Lemmas 4.1, 4.4  imply that $-\Phi$ preserves brackets, and $\Phi$ again turns out to be of form claimed in the theorem.
\end{proof}

All values $a,b$ admitted in Theorem 1.1 and both signs in $\varphi^*\omega=\pm\omega'$ can occur for any $(X,\omega)$, $u\in\cH$ such that $\omega_u$ is analytic, and a suitable $(X',\omega')$.
It suffices to check this when $u=0$.
An example with $\varphi^*\omega=-\omega'$ occurs if we choose $X'$ to be $X$ with the opposite complex structure, so that holomorphic functions in $X$ turn into antiholomorphic functions on $X'$.
The form $\omega'= -\omega$ is K\"ahler on $X'$, $\cH=\cH'$, and $F=\id_{\cH}\colon\cH\to\cH'$ is the isometry sought, with $\varphi=\id_X$.

The generalized Legendre transformation is an example of a local isometry $F$ with $(X,\omega)=(X',\omega')$ and $\varphi=\text{id}_X$, for which $\varphi^*\omega=\omega'$ and $a=-1$, $b=0$, see \cite[Proposition 7.1]{BCKR}. Finally, 
to realize $\varphi^*\omega=\omega'$, $a= 1$, and  $b=0 \text{ or }2/\int_X\omega^n$, we need the Monge--Amp\`ere energy $E\colon\cH\to\bR$ (that goes under other names as well).
Its differential $E_*$ on $T_u\cH$ is given by $E_*\xi=\int_X\xi\omega_u^n$.
If we again choose $(X',\omega')=(X,\omega)$, then $F(u)=u-b E(u)$ will map $\cH$ isometrically to $\cH=\cH'$ (the corresponding $\varphi=\id_X$).

\section{The proof of Theorem 1.3}

For finite dimensional Riemannian manifolds $M,M'$ instead of $\cH,\cH'$ the uniqueness result corresponding to Theorem 1.3 is straightforward.
Since isometries $F\colon M\to M'$ send geodesics to geodesics, the exponential maps at $u\in M$ and $F(u)\in M'$
$$
\exp\colon T_u M\to M,\qquad \exp'\colon T_{F(u)} M'\to M',
$$
diffeomorphisms in neighborhoods of 0 in $T_u M$, resp.~$T_{F(u)}M'$, satisfy $F\circ \exp=\exp'\circ F_*|T_u M$.
This shows that near $u$ at least, $F$ is determined by $F_*|T_u M$.

In $\cH$ this line of reasoning fails because the exponential map is not defined in a neighborhood of $0\in T_u\cH$.
Even if one restricts attention to the space $\cK\subset\cH$ of potentials $u$ with $\omega_u$ analytic, and as in section 7 defines
the natural inductive limit topologies on $\cK$ and $T_u\cK$, there is no reason
why the exponential map should be a local diffeomorphism.
It turns out, nevertheless, that it is possible to modify the finite dimensional argument to apply in $\cH$.
Instead of geodesics we can use arbitrary very analytic curves $u\colon I\to\cH$, and replace uniqueness of geodesics by the fact that such a curve is uniquely determined if it is known how $du/d\tau$  evolves.

To prove Theorem 1.3  we need the notion of parallel transport.
In general infinite dimensional Riemannian manifolds parallel lifts of smooth curves do not necessarily exist and even if they do they may not be unique (cf.~\cite[section 6]{L2}).
Therefore parallel transport along curves cannot be defined in general.
Nevertheless, in $\cH$ parallel transport does exist, and has a simple description.
Let $I\subset\bR$ be an open interval, $0\in I$, and $u\colon I\to\cH$ a $C^1$ curve. It will be convenient to write the map as $\tau\mapsto u_\tau$.
By integrating the time dependent vector field
$$
{1\over 2}\ \text{grad}_{u_\tau} {du_\tau\over d\tau}
$$
on $X$ (gradient with respect to the metric of $\omega_{u_\tau}$), we obtain a $C^1$ family $\varphi_\tau\colon X\to X$ of diffeomorphisms.
Thus
\begin{equation}
{\partial\varphi_\tau\over\partial\tau}={1\over 2} \big(\text{grad}_{u_\tau}{du_\tau\over d\tau}\big)\circ \varphi_\tau,\qquad\varphi_0=\id_X.
\end{equation}
The parallel transport of $\xi\in T_{u_\tau}\cH$ along the curve $u$ is $\xi\circ\varphi_\tau\in T_{u_0}\cH$, see [M, p.~234, or D,S].

We define the tempo of a curve $u\in C^1 (I,\cH)$ as the continuous function $\theta\colon I\to T_{u_0}\cH\approx C^\infty (X)$ given by
\begin{equation}
\theta_\tau={du_\tau\over d\tau}\circ\varphi_\tau,
\end{equation}
the parallel transport of the velocity $du_\tau/d\tau \in T_{u_\tau}\cH$ to $T_{u_0}\cH$. 

From now on we assume $\omega$ is analytic. By Proposition 2.1 this does not restrict the generality.

\begin{lem}Suppose two very analytic curves $u,u'\colon I\to\cH$ start at the same point $u_0=u'_0$, and their tempi $\theta,\theta'\colon I\to C^\infty(X)$ agree.
Then $u=u'$.
\end{lem}

\begin{proof}Since the set $\{\tau\in I\colon u_\tau=u'_\tau\}$ is closed in $I$, it will suffice to show it is open, and in fact only that $u_\tau=u'_\tau$ for $\tau$ in a neighborhood of 0; or equivalently, that knowing $u_0$ and $\theta$, equations (5.1), (5.2) uniquely determine $u_\tau$ for small $\tau$.
Our argument will be local. We choose local coordinates $x_j$ on $X$, write $\omega=i\sum \omega_{jk} dx_j\wedge d\overline x_k$, and
define $v\colon I\times X\to\bR$ and $\psi\colon I\times X\to X$ by
$$
v(\tau,\cdot)=u_\tau,\qquad\psi(\tau,\cdot)=\varphi_\tau.
$$
(5.1) and (5.2) can be rewritten as functional equations
$$
{\partial\psi(\tau,x)\over\partial\tau}=P(\psi,v)(\tau,x),\qquad
\theta_\tau(x)={\partial v\over\partial\tau} (\tau,\psi(\tau,x)),
$$
where $P(\psi,v)$ is an expression involving $\det(\omega_{jk}+v_{x_j\overline x_k})^{-1}$ and various partials of $\psi,v$ that are added, multiplied, and composed with each other.
However, no $\tau$ derivative of $\psi$ and only first $\tau$ derivative of $v$ occur.
Differentiating these equations $k-1$ times we obtain
\begin{eqnarray}
& & {\partial^k\psi(\tau,x)\over\partial\tau^k}=P_k (\psi,v) (\tau,x),\\
& & {\partial^{k-1}\theta_\tau(x)\over\partial\tau^{k-1}}={\partial^k v\over\partial\tau^k} \big(\tau,\psi(\tau,x)\big)+Q_k (\psi,v)(\tau,x).
\end{eqnarray}
Induction shows that $P_k,Q_k$ are composed of various partials of $\psi,v$; but $P_k$ contains $\tau$ derivatives of $\psi$ only of order $<k$ and of $v$ of order $\leq k$, while in $Q_k$, $\tau$ derivatives of both $\psi,v$ are of order $<k$.
Therefore (5.3), (5.4) recursively determine $\partial^kv/\partial\tau^k,\partial^k\psi/\partial\tau^k$ at $\tau=0$, starting with
\begin{equation*}
 v(0,x)=u_0(x),\quad \psi(0,x)=x,\quad
  \frac{\partial v}{\partial\tau}(0,x)=\theta_0(x),\quad
\frac{\partial\psi}{\partial\tau}(0,x)=P_1(\psi,v)(0,x),\ldots
\end{equation*}
In particular, $u_0$ and $\theta$ determine all partials $\partial^k v/\partial\tau^k$ for $\tau=0$, and $v$ being analytic, also $v(x,\tau)$ for small $\tau$, as required.
\end{proof}

\begin{proof}[Proof of Theorem 1.3]
We are considering an analytic $u_0\in\cH$, its connected open neighborhood $\cU\subset\cH$, $C^\infty$ isometries $F,G\colon\cU\to\cH'$ such that $F(u_0)=G(u_0)$ and $F_*|T_{u_0}\cU=G_*|T_{u_0}\cU$.
To show $F$ and $G$ agree it will suffice to show they agree in some neighborhood of $u_0$, for example in one that is convex in $C^\infty(X)$ and on which $F^{-1}\circ G$ is defined.
Given an analytic $u_1$ in such a neighborhood, the curve $u\colon \tau\mapsto u_\tau=(1-\tau) u_0+\tau u_1\in\cH$ is very analytic for $\tau$ in some open interval $I\supset [0,1]$. By Theorem 3.3 its image $u'=F^{-1}\circ G\circ u$ is also very analytic.
Since $u_0=u'_0$ and since parallel transport commutes with $C^2$ isometries, see \cite[ Lemma 6.5]{L2}, the tempi of $\tau\mapsto u_\tau$, $\tau\mapsto u'_\tau$ agree.
Hence $u_\tau=u'_\tau$ by Lemma 5.1, in particular $F(u_1)=G(u_1)$. Analytic $u_1$ being dense, $F=G$ in a neighborhood of $u_0$, and the theorem follows.
\end{proof}

Even with $F,G$ assumed only $C^1$ this proof would go through if Lemma 5.1 could be shown for curves $u,u'$ that are only $C^1$. 
In fact,  showing Lemma 5.1 for $C^\infty$ curves would suffice, because  $C^1$ isometries can be proved to send very analytic curves to smooth curves. 
Although \cite[Lemma 6.5]{L2} cannot be used in this generality, the criterion in \cite[Lemma 7.2]{L2} would show that parallel transport commutes even with $C^1$ isometries.

\section{The proof of Theorem 1.4}

In this section we will construct $C^\infty$ local isometries of $\cU\subset\cH$ into $\cH'$, with given differential $\Phi\colon T_u\cH\to T_{u'}\cH'$, provided $\omega_u$, $\omega_{u'}$ are analytic.
Upon replacing $\omega,\omega'$ by $\omega_u,\omega_{u'}$ we can assume $u,u'=0$.
Then $\omega,\omega'$ will be analytic.
This will free us to use $u,u'$ for other purposes.

We start by recalling the transformations constructed in \cite{L1}.
Let $X,X'$ be arbitrary complex manifolds.
We abbreviate the holomorphic cotangent bundle $T^{*(1,0)} X\to X$ as $\pi\colon T^* X\to X$, and employ $\pi'\colon T^* X'\to X'$ similarly.
There is a tautological $(1,0)$ form $A$ on $T^*X$, whose value on a tangent vector $v\in T_z^{1,0} (T^* X)$ is $A(v)=\langle z,\pi_* v\rangle$, where $z\in T^* X$ and $\langle \ , \rangle$ denotes the pairing between $(1,0)$ forms and $(1,0)$ vectors on $X$.
The canonical holomorphic symplectic $(2,0)$ form on $T^*X$ is $\Omega=-dA$.
Local coordinates $x_j$ on $X$ give rise to local coordinates $x_j,p_j$ on $T^*X$; then $A=\sum p_j dx_j$ and $\Omega=\sum dx_j\wedge dp_j$.
The corresponding forms on $T^*X'$ will be denoted $A',\Omega'$.

Any holomorphic section $g\in\cO(X,T^* X)$, i.e., a holomorphic $(1,0)$ form, determines a biholomorphism
\begin{equation}
T^*X\ni z\mapsto z+g(\pi z)\in T^* X.
\end{equation}
It is straightforward to check, for example in local coordinates, that (6.1) preserves $\Omega$ if $\partial g=0$.

Now suppose $v\in C^\infty(X)$ (real valued!), and consider $\partial v$ as a map $g_v\colon X\to T^* X$.
By  \cite[Proposition 2.2]{L1}
\begin{equation}
g_v^* A=\partial v\quad\text{ and }\qquad g_v^*\Omega=\partial\opartial v.
\end{equation}
Suppose furthermore that we are given a holomorphic symplectomorphism $\Psi$ of a neighborhood $N$ of $g_v(X)$ on an open subset of $T^*X'$.
Since $d(\Psi^* A'-A)=\Omega-\Psi^*\Omega'=0$, locally on $N$ there are holomorphic functions $h$ such that $dh=\Psi^*A'-A$.

\begin{lem}Assume that there is an $h\in\cO(N)$ such that $dh=\Psi^*A'-A$, and that $\psi=\psi_v=\pi'\circ\Psi\circ g_v\colon X\to X'$ is a diffeo\-morphism.
Then $v'\in C^\infty(X)$ defined by
\begin{equation}
v'\circ \psi=v+2\Ree h\circ g_v
\end{equation}
satisfies
\begin{equation}
\psi^*\partial\opartial v'=\partial\opartial v.
\end{equation}
\end{lem}

This is \cite[Theorem 2.1]{L1}.

\begin{lem} Under the assumptions of Lemma 6.1
\begin{gather}
g_{v'}\circ \psi=\Psi\circ g_v \qquad\text{and}\\
\Psi^*\pi'^*\partial v'=A+dh\qquad\text{on}\quad T_{g_v(x)}T^*X,\quad x\in X.
\end{gather}
\end{lem}

Here and later, $T$ indicates {\sl real} tangent bundle.

\begin{proof}Like Lemma 6.1, this too is an exercise in the chain rule. (6.5) is \cite[Proposition 2.4]{L1}. As to (6.6), for any $y\in X'$ and $w\in C^\infty(X')$, the definition of $A'$ yields $\pi'^*\partial w=A'$ on $T_{g_w(y)}T^*X'$. When $y=\psi(x)$ and $w=v'$, this implies by (6.5)
\begin{equation*}
\Psi^*\pi'^*\partial v'=\Psi^*A'=A+dh\qquad\text{on}\quad T_{g_v(x)}T^*X.
\end{equation*}
\end{proof}

Consider next an interval $I\subset\bR$ and a $C^1$ curve $I\ni t\mapsto v_t\in C^\infty(X)$.

\begin{lem}Suppose that $g_{v_t} (X)\subset N$ and $\psi_t=\pi'\circ \Psi\circ g_{v_t}\colon X\to X'$ is a diffeomorphism onto some neighborhood of a fixed open $X^0\subset X'$, for all $t\in I$.
Define the transform $v'_t\in C^\infty (X^0)$ of $v_t$ by (6.3), with $v,v'$ replaced by $v_t,v'_t$ and $\psi$ by $\psi_t$.
Then $(dv'_t/dt)\circ \psi_t=dv_t/dt$ on $\psi_t^{-1} (X_0)$.
\end{lem}

\begin{proof}For brevity, write $g_{v_t}=g_t$.
From $v'_t\circ \psi_t=v_t+2\Ree h\circ g_t$ and $\psi_t=\pi'\circ\Psi\circ g_t$ the chain rule gives, with $x\in\psi_t^{-1} (X^0)$.

\begin{align*}
{dv_t(x)\over dt}-{dv'_t\over dt}(\psi_t(x))&=(dv'_t)(\pi'\circ\Psi)_*\frac{dg_t(x)} {dt}-2(\Ree dh)\big(\frac {dg_t(x)}{ dt}\big)\\
&=\big(\Psi^*\pi'^*dv'_t-2\Ree dh\big)\big(\frac{dg_t(x)}{dt}\big)=2(\Ree A)\big(\frac{dg_t(x)}{dt}\big)=0,
\end{align*}
by (6.6) and because $dg_t/dt$ is vertical, $\pi_* dg_t/dt=0$.
This proves the lemma.
\end{proof}

There are other ways as well to define the transformation $v\mapsto v'$.
One is based on (6.5).
Given $\Psi$ and $v$, this determines $g_{v'}$, i.e.~$\partial v'$, uniquely, hence the real valued $v'$ up to an additive constant---assuming $X$ is connected.
Of course, in formula (6.3) there is ambiguity as well, namely $\Psi$ determines $h$ only up to a locally constant function.
One way to get around this, without mentioning $h$, is to fix a choice of $v'_0$ for some $v_0\in C^\infty(X)$; this then determines $h$ uniquely in a neighborhood of $g_{v_0}(X)$, hence it also determines $v'$ for $v$ close to $v_0$.
Alternatively, given $v_0$ and $v'_0$, we can connect $v_0$ with a nearby $v$ by a $C^1$ curve $[0,1]\ni t\mapsto v_t$, $v_1=v$, and use Lemma 6.3 to define
\begin{equation}
v'=v'_0+\int_0^1\ {dv_t\over dt}\circ \psi_t^{-1} dt,\qquad \psi_t=\pi'\circ \Psi\circ g_{v_t}.
\end{equation}
By virtue of Lemma 6.3 it does not matter which curve $t\mapsto v_t$ we choose, as long as it stays close to $v_0$.

It will be convenient to refer to $v'$ as the $\Psi$ transform of $v$, even though $\Psi$ itself determines $v'$ only up to an additive constant.

\begin{lem}If we use $\Psi'=\Psi^{-1}$ and $h'=-h\circ\Psi^{-1}$ to transform functions in $C^\infty(X')$, then the transform of $v'$ given in (6.3) is $v$ itself.
\end{lem}

\begin{proof}The relation (6.5) is the same as saying that $\Psi(g_v(X))=g_{v'}(X')$, or $g_v(X)=\Psi^{-1} (g_{v'}(X'))$, which in turn is the same as $g_v\circ\psi_{v'}=\Psi^{-1}\circ g_{v'}$.
Comparison with (6.5) shows $\psi_{v'}=\psi_v^{-1}$.
Hence (6.3), with $\psi=\psi_v$, is indeed equivalent to
$$
v\circ \psi_{v'}=v'+2\Ree h'\circ g_{v'}.
$$
\end{proof}

The transformation $v\mapsto v'$ can be modified to transform relative potentials on K\"ahler manifolds $(X,\omega)$.
One possibility would be to cover $X$ with open sets $U$ on which $\omega=i\partial\opartial f_U$ with suitable $f_U\in C^\infty(U)$, transform $u\in\cH$ by the recipe above applied to $v=u+f_U$, and make sure that the transformed potentials agree on overlaps $U\cap V$.
There is a better approach, though, that first appeared in \cite{Se}: the K\"ahler form $\omega$ determines a complex manifold structure on $T^*X$, different from its standard structure.
The new complex manifold comes with a canonical symplectic form like $T^*X$ does, and symplectic biholomorphisms of the deformed cotangent bundles induce transformations of relative potentials, that turn out to be local isometries $\cH\to\cH'$.

It is not surprising that a K\"ahler form, indeed any closed $(1,1)$--form $\omega$ on $X$, can be encoded in a complex structure on $T^*X$.
A $(1,1)$--form $\omega$ on $X$, viewed as a $T^*X$ valued $(0,1)$--form, gives rise to a $(0,1)$--form $\pi^*\omega$ on $T^*X$, valued in $\pi^* T^*X$.
Since $T_x^*X$ can be canonically identified with $T_z(T_x^*X)$ or $T_z^{1,0}(T_x^* X)\subset T^{1,0} (T^*X)$ (here $\pi z=x$), our $\omega$ induces a $(0,1)$ form $\theta$ on $T^*X$, valued in $T^{1,0}(T^*X)$.
Such a form is a deformation tensor, and determines an almost complex structure on $T^*X$, that will be an honest complex structure if $\theta$ satisfies an integrability condition, see \cite{Kd}.
For our $\theta$ integrability follows from $d\omega=0$.
Patyi explains all this in rather greater generality in \cite{P}.
We shall, however, not take this infinitesimal approach, but instead obtain the deformation of $T^*X$ following Semmes' local construction.
A variant appears in \cite{Do2}, where Donaldson associates with $(X,\omega)$ a biholomorphism class of complex manifolds.

So, suppose on our complex manifold $X$ we are given a smooth real $(1,1)$--form $\omega$ with $d\omega=0$.
Assume first that $X$ is simply connected and $\omega=i\partial\opartial f$ with $f\in C^\infty(X)$.
The potential $f$  induces a diffeomorphism
$$
z\mapsto z+(\partial f)(\pi z)
$$
of $T^*X$, and we define a complex manifold $X(f)$ by pulling back the complex structure of $T^*X$ along this diffeomorphism.
This makes the map
\begin{equation}
X(f)\ni z\mapsto z+(\partial f)(\pi z)\in T^*X
\end{equation}
biholomorphic.
The underlying smooth manifolds of $X(f)$ and $T^*X$ agree, and the projection $\pi\colon X(f)\to X$ is holomorphic.
Although the fibers inherit a complex vector space structure from $T^*X$, fiberwise addition is not holomorphic, and $X(f)\to X$ is not a holomorphic vector bundle in general.
Rather, it is an affine bundle for $T^*X$.
This means that there is a holomorphic fiber map $a\colon T^* X\times_X X(f)\to X(f)$ such that with $z\in X(f)$, $x=\pi z$, and $z_1,z_2\in T_x^*X$, 
$$
a(z_1+z_2,z)=a(z_1,a(z_2,z)),\quad\text{and\quad $a(\cdot,z)\colon T^*_xX\to\pi^{-1}x$ is bijective.}
$$
Indeed, $a(w,z)=w+z$ will do.

If $f_1$ is another potential of $\omega$, then $f_1=f+2\Ree F$ with $F\in \cO(X,T^*X)$.
Since $w\mapsto w+(\partial F)(\pi w)$ is a biholomorphism of $T^*X$, the complex structure of $X(f_1)$, pulled back from $T^*X$ along the map $z\mapsto z+(\partial f)(\pi z)+(\partial F)(\pi z)$, agrees with $X(f)$.
In other words, the complex manifold $X(f)$ depends only on $\omega$ and not on the choice of its potential $f$; henceforward we will denote it $X(\omega)$.

The pull back $\Omega(\omega)$ of $\Omega$ along the map (6.8) does not depend on the choice of $f$, either, because the maps (6.1) preserve $\Omega$.
There is no canonical way to pull back the $(1,0)$ form $A$ to $X(\omega)$, though.

It should be clear that even if $\omega$ has no global potential $f$, the holomorphic symplectic manifold structure $(X(\omega),\Omega(\omega))$ can still be defined on the smooth manifold $T^*X$, by requiring that the maps $X(\omega)\ni z\mapsto z+(\partial f)(\pi z)\in T^*X$ should be holomorphic over open subsets $U\subset X$ where $\omega$ has a potential $f$.
We emphasize that points of $X(\omega)$ are still $(1,0)$ forms on $X$; so for example any $u\in C^\infty(X)$ determines a section $g_u$ of $\pi\colon X(\omega)\to X$, with $g_u(x)=\partial u(x)$.
Over an open set where $\omega=i\partial\opartial f$, using (6.2) one computes 
\begin{equation}
g_u^* \Omega(\omega)=g_{u+f}^* \Omega=\partial\opartial (u+f)=-i\omega+\partial\opartial u=-i \omega_u.
\end{equation}

\begin{lem}If $\omega_u$ is nondegenerate, then $g_u(X)\subset X(\omega)$ is totally real.
\end{lem}

Denoting by $J$ the complex structure tensor of $X(\omega)$, that a submanifold $Y\subset X(\omega)$ is totally real means that $TY\cap JTY$ consists of the zero section or, 
equivalently, $(\bC\otimes TY)\cap T^{0,1}X(\omega)\subset\bC\otimes TX(\omega)|Y$ is the zero section.

\begin{proof}Suppose $x\in X$ and $\tau\in\bC\otimes T_x X$ is not zero. 
By  (6.9) $\iota_{g_{u*}\tau}\Omega(\omega)=-\sqrt{-1}\iota_\tau\omega_u\neq 0$.
Since $\Omega(\omega)$ is a $(2,0)$ form, this proves $g_{u*}\tau\notin  T^{0,1}X(\omega)$, and so $g_u(X)$ is indeed totally real.
\end{proof}

Now consider a pair $(X,\omega),\,(X',\omega')$  of compact K\"ahler manifolds.
We will use holomorphic symplectic transformations $X(\omega)\to X'(\omega')$ to produce isometries of K\"ahler potentials, as follows.

Let $N\subset X(\omega)$ be an open neighborhood of the zero section $g_0(X)$, and $\Theta\colon N\to X'(\omega')$ a symplectic biholomorphism on an open subset of $X'(\omega')$, such that $\Theta(g_0(X))$ is the zero section of $X'(\omega')$.
Let furthermore
$$
\cN=\{u\in C^\infty(X)\colon g_u(X)\subset N\}.
$$
For $u\in\cN$ close to 0 in the $C^2$ topology, the map $\theta_u=\pi'\circ \Theta\circ g_u\colon X\to X'$ is a diffeomorphism, a small $C^1$ perturbation of $\theta_0=\pi'\circ \Theta\circ g_0$.
We define the transform $u'$ of $u$ by connecting $u_0=0$ and $u_1=u$ with a $C^\infty$ curve $[0,1]\ni t\mapsto u_t\in\cN$, and in analogy with (6.7), letting 
\begin{equation}
u'=\int_0^1\ {du_t\over dt}\circ \theta_{u_t}^{-1} dt,\qquad \theta_{u_t}=\pi'\circ \Theta\circ g_{u_t}.
\end{equation}

\begin{thm}
The integral in (6.10) does not depend on the curve $t\mapsto u_t$ connecting 0 and $u$, as long as it stays close to $0\in\cN$ (close in the $C^2$ topology).
We have
\begin{equation}
\theta_u^* \omega_{u'}=\omega_u.
\end{equation}
Define a map $F$ in a $C^2$ neighborhood of $0\in\cN$ by $F(u)=u'\in C^\infty(X')$.
This is a $C^\infty$ diffeomorphism onto a $C^2$ neighborhood of $0\in C^\infty(X')$, and its differential acts between $T_u\cN\approx C^\infty(X)$ and $T_{F(u)} C^\infty (X')\approx C^\infty(X')$ by $F_*\xi=\xi\circ\theta_u^{-1}$.
\end{thm}

\begin{proof}Choose $Y\subset X$, $Y'\subset X'$ open so that $\omega|Y=i\partial\opartial f$, $\omega'|Y'=i\partial\opartial f'$ with some $f\in C^\infty(Y)$, $f'\in C^\infty(Y')$; make sure that $Y'$ contains the closure of $\bigcup_{0\leq t\leq 1}\theta_{u_t} (Y)$.
The biholomorphisms $\rho\colon Y(\omega)\to T^* Y$, $\rho'\colon Y'(\omega')\to T^* Y'$,
$$
\rho(z)=z+(\partial f)(\pi z),\qquad \rho'(z')=z'+(\partial f')(\pi' z'),
$$
satisfy $\rho^*\Omega=\Omega(\omega)$, $\rho'^*\Omega'=\Omega'(\omega')$.
Hence $\Psi=\rho'\circ \Theta\circ \rho^{-1}$ is a holomorphic symplectomorphism between open subsets of $(T^*Y,\Omega)$, $(T^*Y',\Omega')$, and can be used to transform $v\in C^\infty(Y)$ as described earlier in this section.
Note that 
\begin{equation}
\psi_{u+f}=\pi'\circ \Psi\circ g_{u+f}=\pi'\circ\Theta\circ g_u=\theta_u
\end{equation}
for $u\in C^\infty(X)$ close to 0.
Since $\Theta$ maps the zero section to the zero section, $g_{f'}\circ\psi_f=\Psi\circ g_f$.
This means that with a unique choice of $h$ satisfying $\Psi^* A'=A+dh$, the $\Psi$--transform of $f$ is $f'$.
Let $v'$ denote the $\Psi$--transform of $v$ close to $f$ (using the same $h$).
Writing $v=u+f$ and $v_t=u_t+f$, by (6.7)
$$
u'=\int_0^1\ {du_t\over dt}\circ \theta_{u_t}^{-1} dt=\int_0^1\ {dv_t\over dt}\circ \psi_{v_t}^{-1}dt=v'-f',
$$
indeed independent of the choice of the path $t\mapsto u_t$.

Next,
\begin{equation}
g_{u'}\circ\theta_u=\Theta\circ g_u\qquad\text{on } Y,
\end{equation}
because  by (6.5) and (6.12)
\begin{equation*}
\rho'\circ g_{u'}\circ\theta_u= g_{v'}\circ \theta_u=g_{v'}\circ \psi_v=\Psi\circ g_v=\rho'\circ\Theta\circ g_u.
\end{equation*}
Hence (6.9), applied twice gives (6.11) over $Y$:
$$
\theta_u^*\omega_{u'}=i\theta_u^* g_{u'}^*\Omega'(\omega')=i g_u^* \Theta^* \Omega'(\omega')=i g_u^*\Omega(\omega)=\omega_u.
$$
But since $X$ can be covered with open $Y$ as above, and the corresponding $Y'$ will cover $X'$, it follows that $u'$ over all of $X'$ is independent of the choice of $u_t$ and (6.11) holds on all of $X$.

That $u'=F(u)$ depends smoothly on $u$ follows from (6.10) if we choose $u_t=tu$.
The inverse of $F$ can be constructed as $F$, except that $\Theta$ has to be replaced by $\Theta^{-1}$; this follows from Lemma 6.4.
Hence $F^{-1}$ is also $C^\infty$.
Finally, to compute the action of $F_*$ on some $\xi\in T_u\cN\approx C^\infty(X)$, we choose the curve $t\mapsto u_t$ in (6.10) so that $du_t/dt=\xi$ when $t=1$.
Then
$$
F_*\xi=\frac d{dt}\Big|_{t=1} F(u_t)={d\over dt}\Big|_{t=1}\int_0^t\ {du_\tau\over d\tau}\circ \theta_{u_\tau}^{-1} d\tau=\xi\circ\theta_u^{-1},
$$
as claimed.
\end{proof}

\begin{proof}[ Proof of Theorem 1.4]
For brevity, write $\{\ , \ \}$ and $\{\ , \ \}'$ for the Poisson brackets $\{\ , \ \}_u$, $\{\ ,\ \}_{u'}$ on $C^\infty(X)$, $C^\infty(X')$.
Upon replacing the K\"ahler form $\omega$ by $\omega_u$ and $\omega'$ by $\omega_{u'}$ we reduce ourselves to the case when $\omega,\omega'$ are analytic and $u,u'=0$.
So we are given a linear isomorphism $\Phi\colon T_0\cH\approx C^\infty(X)\to T_0\cH'\approx C^\infty(X')$ such that 
\begin{gather}
 \{\Phi\xi,\Phi\eta\}'=\Phi\{\xi,\eta\}\qquad\text{or}\qquad\{\Phi\xi,\Phi\eta\}'=- \Phi \{\xi,\eta\},\text{ and}\\
 \int_{X'} (\Phi\xi)^2 {\omega'}^{n}\Big /\int_{X'} {\omega'}^{n} =\int_X\xi^2\omega^n\Big/\int_X\omega^n.
\end{gather}
\cite[(8.10) Theorem]{AG} of Atkin and Grabowski, applied to $\Phi$ or $-\Phi$ gives that $X$ and $X'$ are diffeomorphic, in particular have the same dimension; this in the $n$ in (6.15).

Further reductions are possible.
The Atkin--Grabowski theorem, as discussed in the proof of Theorem 1.1, in addition to a diffeomorphism $\varphi\colon X'\to X$  provides numbers $a,b$ such that
$$
\pm\varphi^*\omega=\omega'\qquad \text{and}\qquad\Phi\xi=a\varphi^*\xi-b\int_X\xi\omega^n.
$$
The isometry condition (6.15) then implies $a=\pm 1$ and $b=0\text{ or }2a/\int_X\omega^n$, as we saw in the proof of Theorem 1.1.
In the paragraph following that proof we pointed out that local isometries can produce any admissible $a,b$ and sign in $\pm \varphi^*\omega=\omega'$.
By composing our $\Phi$ with the differentials of these local isometries we reduce our considerations to the case when $\varphi^*\omega=\omega'$ and $\Phi\xi=\varphi^*\xi$.
Since $\Phi$ maps analytic functions to analytic functions, $\varphi$ is analytic.

Using notation introduced above, the zero section $g_0(X)\subset X(\omega)$ is totally real, as $\omega$ is nondegenerate; see Lemma 6.5.
It is also an analytic submanifold of $X(\omega)$, because in the construction of $X(\omega)$ the local potentials of $\omega$ themselves are analytic.
If we write $0'$ for the zero function in $C^\infty(X')$, then the same holds for the zero section $g_{0'}(X')\subset X'(\omega')$.
It follows that the analytic diffeomorphism
$$
g_{0'}\circ\varphi^{-1}\circ g_0^{-1}\colon g_0(X)\to g_{0'} (X')
$$
extends to a biholomorphic map $\Theta$ between a connected neighborhood $N$ of $g_0(X_0)\subset X(\omega)$ and a neighborhood of $g_{0'} (X')\subset X'(\omega')$.
In fact, $\Theta$ is symplectic between $\Omega (\omega)$ and $\Omega'(\omega')$.
This can be seen by first noting that $(\varphi^{-1})^* g_{0'}^*\Omega'(\omega')=(\varphi^{-1})^*\omega'=\omega$, which implies that for $z\in g_0(X)$ and  $\alpha,\beta\in T_z g_0(X)$
\begin{equation}
(\Theta^* \Omega' (\omega')) (\alpha,\beta)=\Omega (\omega) (\alpha,\beta).
\end{equation}
But, with $J\colon TX(\omega)\to TX(\omega)$ denoting the complex structure tensor, $\Omega(\omega) (J\alpha,\beta)=i\Omega(\omega)(\alpha,\beta)$, and similarly for $\Theta^*\Omega'(\omega')$.
It follows that (6.16) holds for all $\alpha,\beta\in T_z X(\omega)$, $z\in g_0(X)$; whence the holomorphy of $\Omega(\omega)$, $\Omega'(\omega')$, and $\Theta$ implies $\Theta^*\Omega'(\omega')=\Omega(\omega)$ on all of $N$.

Thus we can apply the transformation $u\mapsto u'$ examined in Theorem 6.6, to obtain a $C^\infty$ diffeomorphism $F$ of a $C^2$ neighborhood $\cU$ of $0\in\cH$ on a neighborhood of $0'\in\cH'$.
By Theorem 6.6, if $u\in\cU$ and $\xi\in T_u\cU$
$$
|F_*\xi|^2_{F(u)}=\int_{X'} (\xi\circ\theta_u^{-1})^2 \omega_{F(u)}^{n}=\int_X\xi^2\omega_u^n=|\xi|_u^2,
$$
so $F$ is an isometry.
Finally, the definition of $\Theta$ implies $g_{0'}\circ\varphi^{-1}=\Theta\circ g_0$.
Comparing this with (6.13) gives $\varphi^{-1}=\theta_0$.
Therefore, again by Theorem 6.6, $F_*|T_0\cU=\Phi$ follows.
\end{proof}

\section{Spaces of analytic potentials}

In this section we will discuss what happens with Theorems 1.1 through 1.4 if we choose $\omega,\omega'$ analytic, replace $\cH,\cH'$ by spaces of analytic potentials
$$
\cK=\cH\cap C^\text{an}(X),\qquad \cK'=\cH'\cap C^\text{an}(X') ,
$$
and study local isometries between $\cK$ and $\cK'$.
Although ultimately it will not matter much, still we owe an explanation in what topology to address these local questions.
The topology on $\cK$, say, is inherited from the natural locally convex direct limit topology on $C^{\an}(X)$, and this latter is defined as follows.
We forget the complex structure of $X$, and regard it as a compact real analytic manifold of dimension $m$.
As such, it can be embedded as a totally real, analytic submanifold of an $m$ dimensional complex manifold $X^{\bC}$.
Let $U_j\subset X^{\bC}$ for $j\in\bN$ form a fundamental system of neighborhoods of $X\subset X^{\bC}$.
Any $u\in C^{\an}(X)$ is the restriction of a holomorphic function on some $U_j$, and this holomorphic function can even be taken bounded.
On the space $H^\infty(U_j)$ of bounded holomorphic functions consider the norm $p_j(h)=\sup_{U_j}|h|$, and for an arbitrary sequence $a=(a_j)$ of positive numbers define the norm $p_a$ on $C^{\an}(X)$ by
$$
p_a(u)=\inf \bigl\lbrace\sum^l_1 a_j p_j (h_j)\colon l\in\bN, h_j\in H^\infty(U_j),\text{ and }\sum_1^l h_j|X=u\bigr\rbrace.
$$
The norms $p_a$ for all sequences $a$ define a locally convex topology on $C^{\an}(X)$ and on $\cK$ as well.
Cauchy estimates imply that this topology is finer than the topologies inherited from the $C^k$ topologies on $C^k(X)$.
In particular, $C^k$ neighborhoods are open in $\cK$.
According to K\"othe, the topology on $C^{\an}(X)$ is complete, see \cite{Kt}.

So, what happens with the analytic variants of Theorems 1.1 through 1.4?
Theorem 1.2 becomes tautological, but the rest will stay meaningful and true.
We do not formulate the analytic version of the uniqueness theorem, because its proof would be the same as of the smooth version, with one modification. 
In proving the analog of Theorem 3.3, we would choose a suitable complexification $X^{\bC}$ of $X$ and instead of the space $\cH(\omega_{\bT}\oplus\omega)$  we would work with the space of potentials $v\colon \bT\times X\to \bR$ that have a smooth extension to $\bT\times X^{\bC}$, holomorphic in the $X^{\bC}$ variable.
However, the analytic variants of Theorems 1.1 and 1.4 are not completely obvious, so we discuss them:

\begin{thm}Suppose $\cU\subset\cK$ is open, $F\colon \cU\to\cK'$ is a $C^1$ isometry and $u\in\cU$.
Then $\dim X'=\dim X=n$, and there are an analytic diffeomorphism $\varphi\colon X'\to X$ and real numbers
$a=\pm 1$ and $b=0\text{ or }2a/\int_X\omega^n$ such that
$$
\varphi^*\omega_u=\pm \omega_{F(u)}\qquad\text {and}\qquad F_*\xi=a\varphi^*\xi-b\int_X\xi \omega_u^n\quad\text{for $\xi\in T_u\cK\approx C^{an}(X)$.}
$$

\end{thm}
 
\begin{thm}Suppose for $u\in\cK$, $u'\in\cK'$ we are given an isomorphism $\Phi\colon T_u\cK\to T_{u'} \cK'$ of vector spaces.
If $|\Phi\xi|_{u'}=|\xi|_u$ for all $\xi\in T_u\cK$ and 
$$
\{\Phi\xi,\Phi\eta\}_{u'}=\Phi\{\xi,\eta\}_u\text{ for }\xi,\eta\in T_u\cK\quad\text{or}\quad\{\Phi\xi,\Phi\eta\}_{u'}=-\Phi\{\xi,\eta\}_u\text{ for }\xi,\eta\in T_u\cK,
$$
then there is a $C^\infty$ isometry $F\colon\cU\to\cK'$ of a neighborhood $\cU\subset\cK$ of $u$ such that $F(u)=u'$ and $F_*|T_u\cU=\Phi$.
\end{thm}

The proof of these theorems requires a lemma:

\begin{lem}Suppose $(X,\omega),(X',\omega')$ are compact real analytic symplectic manifolds of dimension $2n$, $2m$, $\{,\}$, $\{,\}'$ the corresponding Poisson brackets, $\Phi\colon C^{\an}(X)\to C^{\an}(X')$ is an isomorphism of vector spaces, and $c>0$ is a constant.
If for all $\xi,\eta,\zeta\in C^{\an}(X)$
\begin{equation}
\int_{X'} (\Phi\xi)^2 {\omega'}^m=c\int_X \xi^2\omega^n\qquad\text{and}\qquad\{\{\Phi\xi,\Phi\eta\}' ,\Phi\zeta\}'=\Phi\{\{\xi,\eta\},\zeta\},
\end{equation}
then $\Phi$ extends to an isomorphism $\Psi\colon C^\infty(X)\to C^\infty(X')$ of Fr\'echet spaces such that
\begin{equation}
\int_{X'} (\Psi\xi)^2 {\omega'}^ m=c\int_X \xi^2\omega^n\qquad\text{and}\qquad\{\{\Psi\xi,\Psi\eta\}' ,\Psi\zeta\}'=\Psi\{\{\xi,\eta\},\zeta\}
\end{equation}
for $\xi,\eta,\zeta\in C^\infty(X)$.
\end{lem}

\begin{proof}Only continuous extension needs justification; (7.2) will then follow since $C^{\an}(X)$ is dense in $C^\infty(X)$.
Fix a finite collection $\Xi\subset C^{\an}(X)$ such that $\sgrad\xi$ for $\xi\in\Xi$ span all tangent spaces $T_x X$ and $\sgrad\Phi\xi$ span all tangent spaces $T_y X'$.
We can define on $C^{\an}(X)$ Sobolev norms of various orders $k$ by
$$
\|\xi\|_k^2=\sum\int_X \{\ldots \{\xi,\xi_1\},\xi_2\},\ldots,\xi_k\}^2\omega^n,
$$
the sum over all choices $\xi_1,\ldots,\xi_k\in\Xi$; and similarly Sobolev norms $\|\quad\|'_k$ on $C^{\an}(X')$, using the $\Phi\xi_j$ instead of $\xi_j$.
The norm $\|\quad\|_k$ is equivalent to any Sobolev $k$--norm defined using local coordinates.
Now (7.1) implies that $\|\Phi\xi\|'_k=\sqrt c\|\xi\|_k$ for even $k$, whence $\Phi$ extends uniquely to an isomorphism $W^k(X)\to W^k (X')$ of Sobolev spaces for all even $k$, and also to an isomorphism $C^\infty(X)\to C^\infty(X')$ of Fr\'echet spaces.
\end{proof}

\begin{proof}[Proof of Theorem 7.1]
As in the proof of Theorem 1.1 in Section 2, $\Phi=F_*|T_u\cU$ will satisfy (7.1) for $\xi,\eta,\zeta\in T_u \cU\approx C^{an}(X)$.
The first relation is the isometry condition, the second comes from isometries preserving curvature.
By Lemma 7.3 $\Phi$ then extends to a linear isomorphism $\Psi\colon C^\infty (X)\to C^\infty (X')$ satisfying (7.2).
The form of $F_*|T_u \cK=\Psi|C^{\an}(X)$ now follows from Lemma 4.4 and the Atkin--Grabowski theorem \cite[(8.10) Theorem]{AG}, as in the proof of Theorem 1.1.
\end{proof}
\begin{proof}[Proof of Theorem 7.2]
Again we extend $\Phi$ to $\Psi\colon T_u\cH\to T_{u'}\cH'$ and again this means that $\Psi$  must be of form
$$
\Psi\xi=a\varphi^*\xi-b\int\xi\omega^n,
$$
with an analytic diffeomorphism $\varphi\colon X'\to X$ satisfying $\varphi^*\omega_u=\pm \omega_{u'}$; $a,b$ subject to the previous restrictions.
Given this, we can follow the proof of Theorem 1.4 in section 6, and construct an isometry between neighborhoods $u\in\cH$, $u'\in\cH'$, whose differential is $\Psi$.
Since the isometry in the proof of Theorem 1.4 maps analytic functions to analytic functions, its restriction to $\cK$ is the $F$ sought.
\end{proof}

\section{Appendix}

In this Appendix we deduce Theorem 4.5 of Cartan from a slightly more general result.
For brevity, we write $[ab]$ for the Lie bracket of elements of a Lie algebra; but the commutator of linear operators we still denote [\,,\,].

\begin{thm}
Let $V$ be a finite dimensional vector space over a field $k$ of characteristic 0 and $[\quad], [\quad]'$ brackets that turn $V$ into Lie algebras.
Suppose $[[ab]c]=[[ab]'c]'$ for all $a,b,c\in V$.
If $(V, [\quad])$ is semisimple, then so is $(V,[\quad]')$, and $V=V_+\oplus V_-$ with $V_\pm$ ideals both for $[\quad]$ and $[\quad]'$, that satisfy $[\quad]'=\pm [\quad]$ on $V_\pm$.
\end{thm}

From this Theorem 4.5 follows if $[\quad]'$ in Theorem 8.1 is taken to be the pullback of $[\quad]'$ of Theorem 4.5 by $f$.

\begin{proof}In the proof we will use standard facts and notation concerning semisimple Lie algebras that can be found in \cite{J}, especially in III.1,5 and IV.
Write $\fg$, $\fg'$ for the Lie algebras $(V, [\quad])$, $(V,[\quad]')$ and 
$\ad_a=[\cdot a]$, $\ad'_a=[\cdot a]'\colon V\to V$.
In general we will append a prime to objects if they refer to $[\quad]'$.
By assumption $\ad_a\ad_b=\ad'_a\ad'_b$; more generally, the product of an even number of $\ad_{a_i}$ is the same as the product of the corresponding $\ad'_{a_i}$.
In particular, the Killing form $B(a,b)=\tr \ \ad_a\ad_b$ in $\fg$ is the same as in $\fg'$.
Since semisimplicity is equivalent to the nondegeneracy of $B$, we see that $\fg'$ is indeed semisimple.

Given $a\in V$, consider the Fitting null component 
$$
F_a=\{x\in V\colon \ad_a^\nu x=0\text{ for some }\nu\in\bN\}=\{x\in V\colon\ad_a^{2\nu}x=0\text{ for some }\nu\in\bN\}
$$
of $\ad_a$.
This is the same as the Fitting null component of $\ad'_a$.
Among the Fitting null components $F_a$ a minimal dimensional being a Cartan subalgebra, there is an $\fh\subset V$ that is a Cartan subalgebra in both $\fg$ and $\fg'$.
We shall first treat the case when $\fh$ splits, meaning that the eigenvalues of $\ad_h$, $\ad'_h$ are in $k$ for $h\in\fh$.
Then $\fg,\fg'$ are split Lie algebras.
This is automatic if $k$ is algebraically closed.
\end{proof}

Any linear form $\alpha$ on $\fh$ determines its corresponding weight space
$$
\fg_\alpha=\{x\in V\colon \ad_h x=\alpha(h)x\text{ for }h\in\fh\},
$$
and $\fg'_\alpha$ is defined similarly.
We have $\fg_0=\fg'_0=\fh$.
The $\alpha$ for which $\fg_\alpha\neq (0)$ are the roots of $\fg$; the nonzero roots form a set $R\subset\fh^*$,
\begin{equation}
V=\fh\oplus \bigoplus_{\alpha\in R} \fg_\alpha,\qquad\text{ and }\qquad\dim \fg_\alpha=1\text{ for }\alpha\in R.
\end{equation}
Since the negative of any root is also a root, the simultaneous eigenspace decomposition of $\ad_h^2$, $h\in\fh$, is
$$
V=\fh\oplus \bigoplus (\fg_\alpha\oplus\fg_{-\alpha}),
$$
the sum over unordered pairs $(\alpha,-\alpha)$ of nonzero roots.
The action of $\ad_h^2$ on $\fg_\alpha\oplus \fg_{-\alpha}$ is multiplication by $\alpha(h)^2$.
Since the simultaneous eigenspace decomposition of ${\ad'_h}^2=\ad_h^2$ is $V=\fh\oplus\bigoplus (\fg'_\beta\oplus\fg'_{-\beta})$, the sum over unordered 
pairs $(\beta,-\beta)$ of nonzero roots of $\fg'$, there is a bijection $\alpha\mapsto\beta$ between nonzero roots of $\fg,\fg'$ so that $\fg_\alpha\oplus
\fg_{-\alpha}=\fg'_\beta\oplus\fg'_{-\beta}$, and furthermore $\alpha^2=\beta^2$, i.e., $\alpha=\pm\beta$.
Hence the roots of $\fg$ and $\fg'$ are the same.

Next we show that for each $\alpha\in R$ either $\fg'_\alpha=\fg_\alpha$ or $\fg'_\alpha=\fg_{-\alpha}$.
Indeed, $[(\fg_\alpha\oplus\fg_{-\alpha})(\fg_\alpha\oplus \fg_{-\alpha})]=[\fg_\alpha \fg_{-\alpha}]$ is spanned by $h_\alpha\in\fh$ defined by $\alpha(h)=B(h_\alpha,h)$, $h\in\fh$, and so is $[\fg'_\alpha\fg'_{-\alpha}]'=[\fg_\alpha\fg_{-\alpha}]'$.
Furthermore $[ef]=B(e,f)h_\alpha$ for $e\in \fg_\alpha$, $f\in\fg_{-\alpha}$.
This means that $\fg_\alpha,\fg_{-\alpha}$ are eigenspaces of all $[\ad_a,\ad_b]=\ad_{[ab]}$ for $a,b\in\fg_\alpha\oplus\fg_{-\alpha}$, with eigenvalues $\pm\alpha ([ab])$.
When $[ab]=h_\alpha$, $\alpha([ab])=B(\alpha,\alpha)\neq -\alpha([ab])$.
To sum up, $\fg_\alpha\cup\fg_{-\alpha}$ consists of all vectors in $\mathfrak g_\alpha\oplus\mathfrak g_{-\alpha}$ that are eigenvectors of all $[\ad_a,\ad_b]=[\ad'_a,\ad'_b]$, $a,b\in\fg_\alpha\oplus \fg_{-\alpha}=\fg'_\alpha\oplus\fg'_{-\alpha}$.
Hence $\fg_\alpha\cup \fg_{-\alpha}=\fg'_\alpha\cup \fg'_{-\alpha}$, and indeed $\fg'_\alpha=\fg_\alpha$ or $\fg_{-\alpha}$.

Let 
$$
P=\{\alpha\in R\colon \fg'_\alpha=\fg_\alpha\},\qquad N=\{\alpha\in R\colon \fg'_\alpha=\fg_{-\alpha}\}.
$$
If $a\in\fg_\alpha$, $b\in\fg_\beta$ then with any $h\in\fh$
$$
\alpha(h)[ab]=[[ha]b]=[[ha]'b]'=\pm\alpha(h)[ab]',
$$
depending on whether $\alpha\in P$ or not. Canceling $\alpha(h)$, and using symmetry as well we obtain
\begin{equation}
[ab]'=\begin{cases}\phantom{-}[ab]\quad\text{if }\alpha,\beta\in P\\-[ab]\quad\text{if }\alpha,\beta\in N\end{cases}\qquad\text{and}\qquad [ab]'=[ab]=0\text{ otherwise.} 
\end{equation}

Write $\alpha\sim\beta$ if $\alpha,\beta\in P$ or $\alpha,\beta\in N$.
Clearly, $\alpha\sim-\alpha$.
Further, if $\alpha\sim\beta$ and $\alpha+\beta\in R$, then $\alpha+\beta\sim\alpha$, because if, say, $\alpha,\beta\in N$, then
$$
\fg_{\alpha+\beta}=[\fg_\alpha \fg_\beta]=[[\fh \fg_\alpha] \fg_\beta]=[[\fh \fg'_{-\alpha}]' \fg'_{-\beta}]'=[\fg'_{-\alpha}\fg'_{-\beta}]'=\fg'_{-\alpha-\beta}.
$$
But conversely, too: if $\alpha,\beta$, and $\alpha+\beta\in R$, then $\alpha\sim\beta$.
For two among $\alpha,\beta,\alpha+\beta$ will be equivalent; if the two are, say, $\alpha$ and $\alpha+\beta$, then $-\alpha\sim\alpha\sim\alpha+\beta$, whence $\beta=\alpha+\beta-\alpha\sim\alpha$.
More generally, it follows by induction that if $\alpha_1+\alpha_2+\ldots+\alpha_m$ is a sum of nonzero roots, and each partial sum is also a nonzero root, then the $\alpha_j$ and the partial sums are all equivalent.

Choose a family of simple roots; this is a basis of $\fh^*$.
Denote by $P_s,N_s$ the collection of simple roots in $P$ resp.~$N$.
Since any nonzero root or its negative can be written as the sum of simple roots, with each partial sum also a root, it follows that roots in $P$, resp.~$N$, are linear combinations of elements of $P_s$, resp.~$N_s$.
Passing to the duals, $\fh=\fh_+ \oplus \fh_-$, where $\fh_{\pm}$ is the linear span of $h_\alpha$, $\alpha\in P$, resp.~$\alpha\in N$.

In light of (8.1) this implies that
$$
V_+=\fh_+ \oplus\bigoplus_{\alpha\in P}\fg_\alpha\qquad\text{and}\qquad V_-=\fh_-\oplus\bigoplus_{\alpha\in N}\fg_\alpha
$$
satisfy $V_+\oplus V_-=V$. Furthermore, by (8.2)
$$
[ab]'=\begin{cases} \phantom{-}[ab]\quad\text{if } a,b\in V_+\\-[ab]\quad\text{if } a,b\in V_-\end{cases}\qquad\text{and}\qquad [ab]'=[ab]=0\text{ if } a\in V_+,\, b\in V_-.
$$
In other words, $\ad'_v=\pm\ad_v$ for $v\in V_\pm$. This property in fact characterizes $V_\pm$:
\begin{equation}
V_{\pm}=\{v\in V\colon \ad'_v=\pm\ad_v\}.
\end{equation}
Indeed, suppose, for instance, that $\ad'_v=\ad_v$ and write $v=v_++v_-$ with $v_\pm\in V_\pm$. Since
$$
\ad_{v_+}+\ad_{v_-}=\ad_v=\ad'_v=\ad'_{v_+}+\ad'_{v_-}=\ad_{v_+}-\ad_{v_-},
$$
$\ad_{v_-}=0$, and $v_-$ is in the center of $\fg$. But the center is trivial since $\fg$ is semisimple; so $v_-=0$ and  $v=v_+\in V_+$ . 

It remains to show that $V_+$ and $V_-$ are ideals both in $\fg$ and $\fg'$.
By symmetry, it will suffice to verify for $\fg$. Let $v\in V_\pm$ and $w\in V$. Then $[vw]'=\pm [vw]$ and
$$
\ad_{[vw]}=[\ad_v,\ad_w]=[\ad'_v,\ad'_w]=\ad'_{[vw]'}=\pm\ad'_{[vw]},
$$
so indeed $[vw]\in V_\pm$.

We are done with the proof if $\fh$ splits. In general,
fix a basis $h_1,\ldots,h_l$ of $\fh$.
Adjoin to the ground field the eigenvalues of $\ad_{h_j}$, $\ad'_{h_j}$ in an algebraic closure of $k$, to obtain a Galois extension $K$ of $k$.
All $\ad_{h_j}$, $\ad'_{h_j}$ are diagonalizable over $K$, and since $\fh$ is abelian, in fact all $\ad_h$, $\ad'_h$, $h\in \fh$ are also diagonalizable.
This means that the Cartan subalgebras $K\otimes\fh$ of the semisimple algebras $K\otimes\fg$, $K\otimes \fg'$ are split (tensor product over $k$).
Therefore what we have proved so far applies to $K\otimes V$:
\begin{equation}
K\otimes V=W_+\oplus W_-,\qquad W_\pm=\{w\colon \ad_w=\pm \ad'_w\},
\end{equation}
ideals in $K\otimes \fg$ as well as in $K\otimes \fg'$, cf. (8.3).

The Galois group $G$ of $K/k$ acts on $K\otimes \fg$, $K\otimes\fg'$ by automorphisms of Lie algebras over $k$.
(8.4) shows that $W_\pm$ are $G$--invariant.
The spaces of fixed vectors of $G$ in $W_\pm$ are
$$
V_\pm=\bigl\lbrace\sum_{g\in G} gw\colon w\in W_\pm\bigr\rbrace.
$$
Clearly, $V_+\oplus V_-\subset K\otimes V$ consists of all fixed vectors of $G$, that is, $V_+\oplus V_-=V$.
Finally, $V_\pm$ are ideals in $\fg$, $\fg'$, for
$$
[V_\pm V]\subset [W_\pm W]\cap V\subset W_\pm\cap V,
$$
and the last term here is the subspace of $W_{\pm}$ fixed by $G$, namely $V_\pm$.

\end{document}